\documentclass[reqno,12pt,letterpaper]{amsart}
\usepackage{amsmath,amssymb,amsthm,graphicx,mathrsfs,url,bbm,array,enumerate}
\usepackage[dvipsnames]{xcolor}
\usepackage[colorlinks=true,linkcolor=Red,citecolor=Green]{hyperref}
\usepackage{tikz-cd}
\usepackage{tikz}
\usepackage{pgfplots}
\pgfplotsset{compat=1.18}
\usepackage{subfigure}
\usepackage{mathabx}
\usepackage{mathtools}

\usepackage[
  backend=biber,
    style=alphabetic,
    sorting=nyt,
    giveninits=true,
    maxnames=10
]{biblatex}
\DeclareFieldFormat{pages}{#1}
\renewbibmacro{in:}{%
  \ifentrytype{article}
    {}
    {\bibstring{in}%
     \printunit{\intitlepunct}}}
\DeclareFieldFormat
  [article,inbook,incollection,inproceedings,patent,thesis,unpublished]
  {title}{\mkbibemph{#1}}
\DeclareFieldFormat{journaltitle}{#1\isdot}
\DeclareFieldFormat[article]{volume}{\mkbibbold{#1}}
\DeclareFieldFormat[article]{number}{\bibstring{number}\addnbspace #1}

\renewbibmacro*{journal+issuetitle}{%
  \usebibmacro{journal}%
  \setunit*{\addspace}%
  \iffieldundef{series}
    {}
    {\newunit
     \printfield{series}%
     \setunit{\addspace}}%
  \printfield{volume}%
  \setunit{\addspace}%
  \usebibmacro{issue+date}%
  \setunit{\addcomma\space}%
 \printfield{number}%
  \setunit{\addcolon\space}%
  \usebibmacro{issue}%
  \setunit{\addcomma\space}%
  \printfield{eid}
  \newunit}

 \DeclareFieldFormat*{title}{#1}
 \DeclareFieldFormat{title}{#1}
 \DeclareFieldFormat
   [article,inbook,incollection,inproceedings,patent,thesis,unpublished]
   {title}{\mkbibquote{#1\isdot}}
 \DeclareFieldFormat
   [suppbook,suppcollection,suppperiodical]
   {title}{#1}

\newcolumntype{L}{>{$}l<{$}}

\def\?[#1]{\textbf{[#1]}\marginpar{\Large{\textbf{??}}}}
\newtheorem*{thm*}{Theorem}
\newtheorem{prop}{Proposition}

\newtheorem{defi}[prop]{Definition}
\newtheorem{lem}[prop]{Lemma}

\numberwithin{equation}{section}
\numberwithin{prop}{section}

\theoremstyle{definition}

\theoremstyle{remark}
\newtheorem*{rem}{Remark}

\renewcommand{\Re}{\mathop{\rm Re}\nolimits}
\renewcommand{\Im}{\mathop{\rm Im}\nolimits}

\DeclareMathOperator{\Op}{Op}

\DeclareMathOperator{\supp}{supp}

\DeclareMathOperator{\WF}{WF}

\definecolor{green}{rgb}{0,0.8,0}

\begin{document}

\title[]{Microlocal analysis and spectral gap \\ for complex hyperbolic manifolds}

\author{Travis Cunningham}
\email{travisdcunningham@gmail.com}

\begin{abstract}
We prove the existence of an essential spectral gap for convex cocompact complex hyperbolic manifolds under the hypothesis that the limit set does not contain a complex circle. We apply Quan's extension to complex hyperbolic manifolds of Vasy's approach to meromorphic continuation of the resolvent, which allows us to adapt the method of Dyatlov--Zahl to prove fine microlocal properties of resonant states. Following recent work of Athreya--Dyatlov--Miller and Dyatlov--J\'ez\'equel, we then take advantage of the different expansion/contraction rates in different directions for the geodesic flow, and apply the one-dimensional Fractal Uncertainty Principle of Bourgain--Dyatlov along the fast expanding/contracting direction. Combining these techniques provides a new resolvent bound for convex cocompact complex hyperbolic manifolds from which the essential spectral gap follows.
\end{abstract}

\maketitle

\section{Introduction}\label{sec:introduction}

Essential spectral gaps are important in mathematical physics due to their implications in the decay of waves on the underlying system. In this paper we study essential spectral gaps for convex cocompact complex hyperbolic manifolds, $(M=\Gamma\backslash \mathbb{CH}^n,g)$. Here $\mathbb{CH}^n$ is the complex hyperbolic space of complex dimension $n\geq 2$ and $\Gamma\subset SU(n,1)$ is a convex cocompact subgroup. See Section~\ref{sec:complex-hyperbolic-manifolds} for precise definitions, and see e.g. \cite{Ka}, \cite{Par}, and \cite{Go} for more on these objects which have fascinating spectral properties and can also be seen as simplified models of certain complexified black holes where a spectral gap would imply the decay of gravitational waves.

The resolvent
\[
R(\lambda):=\left(-\Delta_g-\frac{n^2}{4}-\lambda^2\right)^{-1},
\]
initially defined on $L^2(M)$ in the physical plane $\Im\lambda>n/2$, has meromorphic continuation to all of $\mathbb C$ as an operator from $L^2_c(M)$ to $H^2_{\rm loc}(M)$, with poles of finite rank, see \cite{EMM}, \cite{GuSa}, \cite{quan2021microlocal}. These poles are called resonances, which are the discrete spectral data replacing eigenvalues in systems that allow energy to escape to infinity. We refer to \cite{DyZw} for a thorough overview of the theory of resonances.

Under the hypothesis that the limit set $\Lambda_\Gamma$ does not contain a complex circle, we prove the existence of a resolvent bound and essential spectral gap; that is, we show the existence of $\beta>0$ such that the region $\{\Im\lambda>-\beta\}$ contains at most finitely many resonances. We define and analyze the limit set and complex circles in Section~\ref{sec:complex-hyperbolic-manifolds}. In particular, we characterize the occurrence of complex circles within the limit set in terms of the dynamics of the system, which is an essential component of the proof.

We now state the main result of this paper:

\begin{thm*}
Let $M=\Gamma\backslash \mathbb{CH}^n$ be a convex cocompact complex hyperbolic manifold such that the limit set $\Lambda_\Gamma$ does not contain a complex circle. Then there exist $C_0,\beta>0$ such that for each $\varepsilon>0$ and each $\chi\in C_c^\infty(M)$, we have
\begin{equation}
\|\chi R(\lambda)\chi\|_{L^2(M)\to L^2(M)}\leq C_{\chi,\varepsilon}|\lambda|^{-1-\frac43\min\{0,\Im\lambda\}+\varepsilon},\qquad |\lambda|>C_0,\ \Im\lambda>-\beta.\label{eq:1.1}
\end{equation}
In particular, $M$ has an essential spectral gap of size $\beta$.
\end{thm*}

Essential spectral gaps have a long, rich history, mostly in the setting of (real) hyperbolic surfaces; see \cite{Pat}, \cite{Su}, \cite{Na}, \cite{BoDy17}, \cite{BoDy18}, \cite{DyZa}, \cite{So}, and references therein. More recently, Tao \cite{Ta} gave a spectral gap for the more general case of asymptotically hyperbolic surfaces. This is done using the result of \cite{Vac}, which proves a spectral gap for scattering by several strictly convex obstacles in the Euclidean plane. Most of these results use a powerful tool called the Fractal Uncertainty Principle -- see \cite{Dy19b} for an introduction -- which we also use to prove our Theorem.

Spectral gaps for higher dimensional spaces are much more rare, but in \cite[Theorem A.1]{Ta}, Tao proves a result related to our Theorem but for real hyperbolic manifolds when the limit set does not contain a round (real) circle. This result is proved using the method of Dyatlov--Zahl \cite{DyZa} to reduce the problem to a Fractal Uncertainty Principle (FUP), and then applying the breakthrough result of Cohen \cite{Co} giving an FUP in higher dimensions (see also \cite{JZZ} for a quantitative version of \cite{Co}). Despite similarities in the statements of \cite[Theorem A.1]{Ta} and our Theorem, we remark that both the results and proofs are quite different. On the one hand, complex circles do not arise in the real hyperbolic case, whereas real circles may arise in our setting but pose no problem in the proof of a spectral gap.

Our approach also reduces the proof to a Fractal Uncertainty Principle. However, since \cite{DyZa} does not readily apply to the complex hyperbolic setting, a significant portion of our work in this paper is to adapt the microlocal framework of \cite{DyZa}, as refined by \cite{Dy19a}, to our setting. This requires the use of results of Quan \cite{quan2021microlocal}, which provides an alternative approach to meromorphic continuation of $R(\lambda)$ analogous to the work of Vasy \cite{Vas13a}, \cite{Vas13b} in the real hyperbolic setting (see also \cite{Zw16} and \cite[Chapter 5]{DyZw} for simplified treatments of Vasy's method). These papers allow powerful microlocal analytical methods to be applied to a certain modified version of the scattering resolvent. In the same way that \cite{DyZa}, \cite{Dy19a} use \cite{Vas13a}, \cite{Vas13b} to prove fine microlocal estimates on resonant states in the real hyperbolic setting, we use \cite{quan2021microlocal} to adapt \cite{DyZa}, \cite{Dy19a} to prove similar microlocal estimates in the present setting.

Moreover, for complex hyperbolic quotients, the geodesic flow expands/contracts at different rates along different directions in the unstable/stable decomposition of the tangent spaces, see \eqref{eq:2.9}, \eqref{eq:2.10}. This precludes the use of the FUP of \cite{Co} since the line and ball porosity hypothesis is difficult to verify under such conditions; however, using the approach of Athreya--Dyatlov--Miller \cite{ADM}, which in turn uses ideas of Dyatlov--J\'ez\'equel \cite{DyJe}, we remedy this issue by exploiting the different expansion/contraction rates and applying the one-dimensional FUP of Bourgain--Dyatlov \cite{BoDy18} along the fast unstable/stable directions. This application of the approach of \cite{ADM} requires the use of an anisotropic pseudodifferential calculus developed in \cite[Appendix]{DyJi}, which generalizes the one originating in \cite{DyZa}. In Section~\ref{subsec:anisotropic-calculus} we prove several new results for this calculus that have analogues in \cite{DyZa}, and at the same time that we adapt the methods of \cite{DyZa}, \cite{Dy19a} to the complex hyperbolic setting using \cite{quan2021microlocal}, we also adapt them to this generalized calculus of \cite{DyJi}. We note that the slight improvement in our resolvent bound as compared with the version for real hyperbolic manifolds given in \cite{Ta}, namely our factor of $4/3$ on $-\min\{0,\Im\lambda\}$ rather than $2$, is due to this specific way we adapt \cite{DyZa}, \cite{Dy19a} with the application of the approach of \cite{ADM} in mind.

We now discuss known examples of convex cocompact complex hyperbolic manifolds whose limit sets contain complex circles. Recent work of \cite{LeOh} shows that any such complex circle must be contained in a totally geodesic, properly immersed submanifold in the convex core of $M$, and that there may be at most finitely many such submanifolds in the convex core. Together, these results place restrictions on the existence and prevalence of complex circles in the limit set. Explicit examples include the so-called complex Fuchsian groups and their higher dimensional analogues, see e.g. \cite[Example 8.8]{Ka}, whose limit sets are exactly a complex circle or a sphere as the boundary of a complex $k$-plane. Non-Fuchsian examples include \cite[Theorem D]{BaKi}. Note that finitely many immersed, totally geodesic submanifolds in the convex core does not imply finitely many complex circles in the limit set. The complex Fuchsian groups can be shown to satisfy a spectral gap via the method of separation of variables analogously to \cite[Section 3.3]{JZZ}, but it is an open question whether an essential spectral gap exists for the more complicated examples above.

Finally, we mention several other recent applications of the Fractal Uncertainty Principle. For support properties of semiclassical defect measures, see \cite{DyJi}, \cite{DJN}, \cite{DyJe}, \cite{KiMi}, \cite{ADM}. For Strichartz estimates on asymptotically hyperbolic surfaces, see \cite{HSTZ}. And for improved fractal Weyl bounds matching improved spectral gaps given by the FUP, see \cite{Cu}.

\textbf{Acknowledgement.} We wish to thank Zhongkai Tao for several fruitful discussions, and for introducing us to many useful papers relevant to this one. We also thank him for allowing us to use some of his notes near the end of Section~\ref{sec:complex-hyperbolic-manifolds}, for typing the original handwritten version of this manuscript, and for comments that improved the exposition. We also thank Jialun Li for helping us understand complex circles.

\section{Complex hyperbolic manifolds}\label{sec:complex-hyperbolic-manifolds}

\subsection{Complex hyperbolic space and convex cocompact complex hyperbolic manifolds.}\label{subsec:complex-hyperbolic-space}
We begin by describing two important models of the complex hyperbolic space $\mathbb{CH}^n$ of complex dimension $n\geq2$. We follow portions of \cite[Section 2]{ADM}, \cite{Ka}, \cite{Par}, and refer to those papers along with \cite{Go}, \cite{EMM} for more details.

The \emph{hyperboloid} model is defined as follows: let $\mathbb C^{n,1}:=\mathbb C^{n+1}$ be complex Minkowski space with the Hermitian form
\begin{equation}
\langle z,w\rangle:=-z_0\overline{w}_0+\sum_{j=1}^n z_j\overline{w}_j.\label{eq:2.1}
\end{equation}
where an element $z\in\mathbb C^{n,1}$ is written $z=(z_0,z')$ with $z_0\in\mathbb C$, $z'\in\mathbb C^n$. Motivated by special relativity, the following subsets are called the sets of time-like, light-like, and space-like vectors, respectively:
\begin{equation}
\begin{aligned}
A_-&=\{z\in\mathbb C^{n,1}:\langle z,z\rangle<0\},\\
A_0&=\{z\in\mathbb C^{n,1}\setminus\{0\}:\langle z,z\rangle=0\},\\
A_+&=\{z\in\mathbb C^{n,1}:\langle z,z\rangle>0\}.
\end{aligned}\label{eq:2.2}
\end{equation}
The projection map $\mathbb P:\mathbb C^{n,1}\setminus \{z_0=0\}\to\mathbb C^n$ is defined by
\[
\mathbb Pz=\frac{z'}{z_0}.
\]
Then the complex hyperbolic space $\mathbb{CH}^n$ can be identified with $\mathbb P A_-$, which is the collection of all negative lines in $\mathbb C^{n,1}$. Notice that, in this model, the boundary $\partial\mathbb{CH}^n$ is identified with $\mathbb P A_0$. If we write $\mathbb{C}\mathbb{S}^{n,1}=\{z\in\mathbb C^{n,1}:\langle z,z\rangle=-1\}$, then this is equivalent to writing
\begin{equation}
\mathbb{CH}^n=\mathbb{C}\mathbb{S}^{n,1}/U(1),\label{eq:2.3}
\end{equation}
where $U(1)=\{e^{i\theta}:\theta\in\mathbb R\}$ acts by isometry on $\mathbb{C}\mathbb{S}^{n,1}$ by $e^{i\theta}\cdot z=e^{i\theta}z$. The complex hyperbolic metric is the Riemannian metric induced by the Lorentzian metric $\Re\,\langle\cdot,\cdot\rangle$ on $\mathbb{C}\mathbb{S}^{n,1}$.

In the \emph{unit ball} model we identify $\mathbb{CH}^n$ with the complex ball $\mathbb B^n:=\{z\in\mathbb C^n:|z|<1\}$. Notice that this can be obtained by considering the section of $\mathbb C^{n,1}$ on which $z_0=1$, and then taking those vectors $z=(1,z')$ for which $\langle z,z\rangle<0$. In other words, $\mathbb B^n$ is identified with $A_-\cap\{z_0=1\}$. The metric is given by $-4\partial\bar\partial\log(\rho)$ where $\rho:=1-|z|^2$. In this model the boundary $\partial\mathbb{CH}^n$ is equivalent to the real sphere $\mathbb{S}^{2n-1}$ given by $\{z\in\mathbb C^n:|z|=1\}$.

In either model, the holomorphic sectional curvature is $-1$ and $\mathbb{CH}^n$ is a K\"ahler manifold. Each model has its advantages and disadvantages. We will mostly use the hyperboloid model, but the unit ball model is useful for describing the limit set and complex circles since it treats the boundary uniformly.

Let $SU(n,1)$ denote the Lie group of complex linear automorphisms of $\mathbb C^{n,1}$ with determinant 1 and preserving the Hermitian form \eqref{eq:2.1}. The description of both the hyperboloid and unit ball models in terms of $A_-\subset\mathbb C^{n,1}$ provides a natural description of the action of $SU(n,1)$ on $\mathbb{CH}^n$. Just as is the case for the automorphism group in the real hyperbolic setting, any element of $SU(n,1)$ can be classified as either elliptic, parabolic, or hyperbolic, see e.g., \cite[Section A.2]{EMM}. Given a discrete subgroup $\Gamma\subset SU(n,1)$, we define the limit set $\Lambda_\Gamma\subset\partial\mathbb{CH}^n$ to be the set of accumulation points of any orbit $\Gamma z$, $z\in\mathbb{CH}^n$; the limit set is independent of the choice of $z$. A complex hyperbolic manifold is the quotient $M=\Gamma\backslash\mathbb{CH}^n$ of $\mathbb{CH}^n$ by a discrete subgroup $\Gamma$. If the convex core $\Gamma\backslash\operatorname{convex-hull}(\Lambda_\Gamma)$ is a compact set, then the subgroup $\Gamma$ is called convex cocompact, and the manifold $M=\Gamma\backslash\mathbb{CH}^n$ is called a convex cocompact complex hyperbolic manifold.

We now discuss complex circles, also called $\mathbb C$-circles, in the unit ball model of $\mathbb{CH}^n$. A complex geodesic in $\mathbb{CH}^n$ is the projection of an indefinite complex 2-plane in $\mathbb C^{n,1}$. The boundary of a complex geodesic is called a \emph{complex circle}, see \cite[Section 2]{Ka}, \cite[Sections 5.2 and 5.5]{Par}.

More explicitly, let $z,\zeta\in\mathbb C^{n,1}$ be linearly independent and set
\begin{equation}
B_{z,\zeta}:=\{\eta=\alpha z+\beta\zeta:\alpha,\beta\in\mathbb C\}.\label{eq:2.4}
\end{equation}
The intersection of this indefinite complex 2-plane with $\{\eta_0=1\}$ is a complex line in $\mathbb C^n$, and provided it intersects $\mathbb B^n$, it projects to a complex geodesic in $\mathbb{CH}^n$. Thus, its intersection with the boundary
\[
B_{z,\zeta}\cap\{\eta_0=1\}\cap \mathbb{S}^{2n-1}
\]
is a complex circle. In the next subsection we will connect this discussion of complex circles to the dynamics of the system.

\subsection{Dynamics and the fast/slow, stable/unstable decomposition.}\label{subsec:dynamics}
We now describe the basic properties of the geodesic flow for convex cocompact complex hyperbolic manifolds. In this subsection, we primarily work in the hyperboloid model of $\mathbb{CH}^n$ and will follow portions of \cite[Section 2]{ADM}.

By \eqref{eq:2.3}, the unit tangent bundle $S\mathbb{CH}^n$ can be identified with $S\mathbb{C}\mathbb{S}^{n,1}/U(1)$ where
\begin{equation}
S\mathbb{C}\mathbb{S}^{n,1}=\{(z,\zeta)\in\mathbb C^{n,1}\times\mathbb C^{n,1}:\langle z,z\rangle=-1,\ \langle z,\zeta\rangle=0,\ \langle\zeta,\zeta\rangle=1\}.\label{eq:2.5}
\end{equation}
Using the complex hyperbolic metric, we may identify the cosphere bundle $S^*\mathbb{CH}^n$ with $S\mathbb{CH}^n$, and to ease notation we will often write points of either space as $(z,\zeta)$ with the understanding that all subsequent operations are equivariant under the action of $U(1)$. The tangent and cotangent spaces $T\mathbb{CH}^n\setminus0$ and $T^*\mathbb{CH}^n\setminus0$ are naturally defined through positive homogeneous extension in the second variable.

One additional space we will need is the coframe bundle $F^*\mathbb{CH}^n$. Elements of $F^*\mathbb{CH}^n$ may be written as $(z,\zeta^1,\ldots,\zeta^n)$ with $z\in\mathbb{CH}^n$ and the $\zeta^1,\ldots,\zeta^n\in T_z^*\mathbb{CH}^n$ forming a positively oriented orthonormal basis. Note that we have the natural submersion
\begin{equation}
\pi_S:F^*\mathbb{CH}^n\to S^*\mathbb{CH}^n,\qquad (z,\zeta^1,\ldots,\zeta^n)\mapsto(z,\zeta^1).\label{eq:2.6}
\end{equation}

For $j,k\in\{0,\ldots,n\}$, we let $E_{j,k}$ be the matrix with $(E_{j,k})_{j,k}=1$ and all other entries 0. We define
\[
X=E_{01}+E_{10},\qquad V^\pm:=i(E_{00}\mp E_{01}\pm E_{10}-E_{11}),
\]
\[
W_j^\pm:=E_{0j}\pm E_{1j}+E_{j0}\mp E_{j1},\qquad Z_j^\pm:=i(E_{0j}\pm E_{1j}-E_{j0}\pm E_{j1}).
\]
Then
\begin{equation}
[X,V^\pm]=\pm2V^\pm,\qquad [X,W_j^\pm]=\pm W_j^\pm,\qquad [X,Z_j^\pm]=\pm Z_j^\pm.\label{eq:2.7}
\end{equation}
The flows $\varphi^t:=e^{tX}: S\mathbb{CH}^n\to S\mathbb{CH}^n$ and $e^{sV^\pm}: S\mathbb{CH}^n\to S\mathbb{CH}^n$ are the geodesic flow and the fast horocyclic flows, respectively. Both have natural extensions to $F^*\mathbb{CH}^n$, $S^*\mathbb{CH}^n$, and hence to $T\mathbb{CH}^n\setminus0$ and $T^*\mathbb{CH}^n\setminus0$ by homogeneity. A calculation shows that $\varphi^t$ acts on $F^*\mathbb{CH}^n$ by (note the matrices act on the right)
\[
\varphi^t(z,\zeta^1,\ldots,\zeta^n)=(z\cosh t+\zeta^1\sinh t,z\sinh t+\zeta^1\cosh t,\zeta^2,\ldots,\zeta^n).
\]
The fast horocyclic flows act on $F^*\mathbb{CH}^n$ by
\begin{equation}
e^{sV^\pm}(z,\zeta^1,\ldots,\zeta^n)=\big(z + is(z\pm\zeta^1),\zeta^1+is(\mp z-\zeta^1),\zeta^2,\ldots,\zeta^n\big).\label{eq:2.8}
\end{equation}

We have the following $\varphi^t$-invariant decomposition of the tangent bundle to $S\mathbb{CH}^n$:
\begin{equation}
T(S\mathbb{CH}^n)=\mathbb RX\oplus E_u\oplus E_s,\qquad E_u=\mathbb RV^-\oplus E^-,\qquad E_s=\mathbb RV^+\oplus E^+,\label{eq:2.9}
\end{equation}
where $\mathbb RV^-$ and $\mathbb RV^+$ are the fast unstable and fast stable subbundles, respectively, and $E^-$ and $E^+$ are the slow unstable and slow stable subbundles, respectively. We refer to \cite[Section 2.2.1]{ADM} for a further breakdown of the slow subbundles $E^\pm$ in terms of the vector fields $W_j^\pm,Z_j^\pm$ above. By \eqref{eq:2.7}, $\varphi^t$ preserves the decomposition \eqref{eq:2.9} and, following \cite[Section 2.2.1]{ADM}, we may fix a Riemannian metric on $S\mathbb{CH}^n$ such that the following expansion/contraction property holds for any point $q\in S\mathbb{CH}^n$:
\begin{equation}
|d\varphi^t(q)w|=
\begin{cases}
e^{\mp2t}|w|,&w\in\mathbb RV^\pm(q),\\
e^{\mp t}|w|,&w\in E^\pm(q).
\end{cases}\label{eq:2.10}
\end{equation}
Thus, \eqref{eq:2.10} says that the flow expands/contracts on $\mathbb RV^\pm$ twice as fast as on $E^\pm$.

Extending the spaces $E_u,E_s,E^\pm$, and vector fields $X,V^\pm$ from $S\mathbb{CH}^n\simeq S^*\mathbb{CH}^n$ to $T^*\mathbb{CH}^n\setminus0$ by making them positively homogeneous -- that is, equivariant under the dilation $(z,\zeta)\mapsto(z,\tau\zeta)$ for $\tau>0$ -- we extend the decomposition in \eqref{eq:2.9} to the cotangent bundle. Indeed, letting $\zeta\cdot\partial_\zeta$ be the vector field on $T^*\mathbb{CH}^n$ which is the generator of dilations on the fibers, we have
\begin{equation}
T(T^*\mathbb{CH}^n\setminus0)=\mathbb R(\zeta\cdot\partial_\zeta)\oplus\mathbb RX\oplus E_u\oplus E_s.\label{eq:2.11}
\end{equation}

Next we record two lemmas from \cite{ADM} for future use. The first provides the existence of symplectic coordinates straightening out the decomposition \eqref{eq:2.11} at a point.

\begin{lem}[{\cite[Lemma 2.4]{ADM}}]\label{obj:2.1}
Fix $q^0\in T^*\mathbb{CH}^n\setminus0$. Then there exists a neighborhood $U_0$ of $q^0$ in $T^*\mathbb{CH}^n$ and a symplectomorphism onto its image $\varkappa_0:U_0\to T^*\mathbb R^{2n}$, such that, denoting by $(y_1,\ldots,y_{2n})$ the coordinates on $\mathbb R^{2n}$ and by $(\eta_1,\ldots,\eta_{2n})$ the corresponding coordinates on the fibers of $T^*\mathbb R^{2n}$, we have
\begin{equation}
\begin{gathered}
\varkappa_0(q^0)=0,\\
d\varkappa_0(q^0)(V^+(q^0))\in\mathbb R\partial_{y_1},\\
d\varkappa_0(q^0)(V^-(q^0))\in\mathbb R\partial_{\eta_1},\\
d\varkappa_0(q^0)(E^+(q^0))=\operatorname{span}(\partial_{y_2},\ldots,\partial_{y_{2n-1}}),\\
d\varkappa_0(q^0)(E^-(q^0))=\operatorname{span}(\partial_{\eta_2},\ldots,\partial_{\eta_{2n-1}}),\\
d\varkappa_0(q^0)(X(q^0))\in\mathbb R\partial_{y_{2n}},\\
d\varkappa_0(q^0)(\zeta\cdot\partial_\zeta(q^0))\in\mathbb R\partial_{\eta_{2n}}.
\end{gathered}\label{eq:2.12}
\end{equation}
\end{lem}

As noted in \cite[Section 2.2.4]{ADM}, if we let $V_\perp^\pm$ denote the complement of $\mathbb RV^\pm$, then Lemma~\ref{obj:2.1} implies that
\begin{equation}
d\varkappa_0(q^0)V_\perp^+(q^0)=\ker dy_1,\label{eq:2.13}
\end{equation}
\begin{equation}
d\varkappa_0(q^0)V_\perp^-(q^0)=\ker d\eta_1.\label{eq:2.14}
\end{equation}

Now let $M=\Gamma\backslash\mathbb{CH}^n$ be a convex cocompact complex hyperbolic manifold (see Section~\ref{subsec:complex-hyperbolic-space}). Then the flows $\varphi^t,e^{sV^\pm}$ as well as the decompositions \eqref{eq:2.9},\eqref{eq:2.11} extend to $M$ via the covering map
\[
\pi_\Gamma:T^*\mathbb{CH}^n\to T^*M.
\]
The following lemma is stated in \cite[Lemma 2.5]{ADM} for a compact complex hyperbolic manifold, and we are extending it to the convex cocompact case. Note that it is through this lemma, which describes the propagation of rectangles in the direction of $V^\pm$ and $V_\perp^\pm$, and its use in Section~\ref{sec:reduction-fup}, that the different expansion rates for the flow is utilized; it is the crucial step allowing us to prove porosity for the supports of the symbols constructed in Section~\ref{sec:reduction-fup} and apply the Fractal Uncertainty Principle. We remark that the only difference between the statement in \cite[Lemma 2.5]{ADM} and our Lemma~\ref{obj:2.2} is the assumption that $\alpha e^t<1$ prior to listing the conclusion. This is necessary to modify their proof to the non-compact case.

\begin{lem}[{\cite[Lemma 2.5]{ADM}}]\label{obj:2.2}
Assume that $q^0\in T^*M\setminus0$, $U_0$ is an open set containing $q^0$, and $\varkappa_0:U_0\to T^*\mathbb R^{2n}$ is a diffeomorphism onto its image satisfying the properties \eqref{eq:2.12} and \eqref{eq:2.13},\eqref{eq:2.14}. Take small $\alpha>0$ and two numbers $y_1^0,\eta_1^0\in[-\alpha,\alpha]$, and define the slow unstable/stable rectangles (which are subsets of $T^*M\setminus0$)
\[
R^-_{q^0,\eta_1^0,\alpha}:=\varkappa_0^{-1}\big(\{(y,\eta):|y|+|\eta|\leq\alpha,\ |\eta_1-\eta_1^0|\leq\alpha^2\}\big),
\]
\[
R^+_{q^0,y_1^0,\alpha}:=\varkappa_0^{-1}\big(\{(y,\eta):|y|+|\eta|\leq\alpha,\ |y_1-y_1^0|\leq\alpha^2\}\big).
\]
Then there exists a constant $C$ independent of $\alpha,y_1^0,\eta_1^0$ such that, denoting by $\operatorname{diam}$ the diameter of a subset of $T^*M$, we have for all $t\geq0$ with $\alpha e^t<1$,
\[
\operatorname{diam}\varphi^t(R^-_{q^0,\eta_1^0,\alpha})\leq C\alpha e^t,
\]
\[
\operatorname{diam}\varphi^{-t}(R^+_{q^0,y_1^0,\alpha})\leq C\alpha e^t.
\]
\end{lem}

The proof is identical to the proof of \cite[Lemma 2.5]{ADM}, except the final sentence of that proof does not hold since $M$ is not compact. Instead, we use our hypothesis $\alpha e^t<1$ to obtain the conclusion.

We conclude this section by proving two additional lemmas that we use in Section~\ref{sec:reduction-fup} to obtain the Theorem. The first is similar to \cite[Lemma 3.1]{ADM} but stated for convex cocompact complex hyperbolic manifolds as opposed to the compact complex hyperbolic case considered there.

Recall that for convex cocompact complex hyperbolic $M$, we have defined the fast horocyclic flows $e^{sV^\pm}: S^*M\to S^*M$. Similarly to \cite[Section 3.1]{ADM} and \cite[Definition 2.2]{KiMi}, we have

\begin{defi}\label{obj:2.3}
Let $V\in\{V^+,V^-\}$ and $e^{sV}:S^*M\to S^*M$ be the fast horocyclic flow. Then:
\begin{enumerate}[i)]
\item For $T>0$, a $V$-segment of length $T$ is a set of the form $\{e^{sV}(q):0\leq s\leq T\}$ where $q\in S^*M$.
\item A $V$-orbit is a set of the form $\{e^{sV}(q):s\in\mathbb R\}$, where $q\in S^*M$.
\item A set $U\subset S^*M$ is called $V$-dense if it intersects every $V$-orbit.
\end{enumerate}
\end{defi}

The following lemma, similar to \cite[Lemma 3.1]{ADM} (see also \cite[Section 2.2]{KiMi}) but stated for our noncompact $M$, will allow us to select appropriate cutoffs for constructing our symbols in the proof of the spectral gap. We have in mind the case when $F=K\cap\{|\zeta|_g=1\}$ where $K$ is the trapped set (see Section~\ref{subsec:modified-spectral-family}).

\begin{lem}\label{obj:2.4}
Let $V\in\{V^+,V^-\}$, let $F$ be a compact set in $S^*M$, and assume $F$ contains no $V$-orbit. Then there exists an open set $U\subset S^*M$ such that
\begin{enumerate}[i)]
\item $S^*M\setminus U$ is a compact set containing $F$ in its interior.
\item $U$ is $V$-dense.
\item There exists $T>0$ such that each $V$-segment of length $T$ intersects $U$.
\end{enumerate}
\end{lem}

\begin{proof}
By hypothesis, $S^*M\setminus F$ is $V$-dense and open. Let $U_1\subset U_2\subset\cdots$ be a nested sequence of open sets such that
\[
\overline{U_j}\subset S^*M\setminus F,\qquad S^*M\setminus F=\bigcup_{j\geq1}U_j.
\]
Since $S^*M\setminus F$ is $V$-dense, we have $S^*M=\bigcup_{j\geq1}\widehat U_j$, where
\[
\widehat U_j:=\bigcup_{s\in\mathbb R}e^{sV}(U_j).
\]
Now let $\Omega\subset S^*M$ be an open set such that $F\subset\Omega$ and $\overline\Omega$ is compact.
Since $\overline\Omega\subset S^*M=\bigcup_{j\geq1}\widehat U_j$, there exists $j$ such that $\overline\Omega\subset\widehat U_j$.

Next, we let $W$ be an open set defined as follows: Let $A$ be an open set with $F\subset A$ and $\overline A\subset\Omega$. Set
\[
W:=S^*M\setminus\overline A.
\]
Finally, we define
\[
U':=U_j\cup W.
\]
Then $\overline{U'}\cap F=\varnothing$ and
\[
\bigcup_{s\in\mathbb R}e^{sV}(U')=\left(\bigcup_{s\in\mathbb R}e^{sV}(U_j)\right)\cup\left(\bigcup_{s\in\mathbb R}e^{sV}(W)\right)\supset\overline\Omega\cup W=S^*M.
\]
Thus, $U'$ is $V$-dense. If it does not already, we may add a set to $U'$ to form $U$ satisfying (i) and (ii).

To prove (iii), we note that $S^*M\setminus U$ is compact and
\[
S^*M\setminus U\subset S^*M=\bigcup_{T=1}^\infty\widetilde U_T
\]
where
\[
\widetilde U_T:=\bigcup_{|s|\leq T/2}e^{sV}(U).
\]
Since $\widetilde U_T$ is a nested sequence of open sets, there is $T$ such that $S^*M\setminus U\subset\widetilde U_T$. This proves (iii).
\end{proof}

Finally, we wish to make a connection between $V^\pm$-orbits and complex circles. We proceed similarly to \cite[Section A.1]{Ta}. Let $(z,\zeta^1,\ldots,\zeta^n)\in F^*\mathbb{CH}^n$; in particular, $\pi_S(z,\zeta^1,\ldots,\zeta^n)=(z,\zeta^1)$ (see \eqref{eq:2.6}) so we have
\begin{equation}
\langle z,z\rangle=-1,\qquad \langle z,\zeta^1\rangle=0,\qquad \langle\zeta^1,\zeta^1\rangle=1.\label{eq:2.15}
\end{equation}
By \eqref{eq:2.8} we have
\[
\pi_S\big(e^{sV^\pm}(z,\zeta^1,\ldots,\zeta^n)\big)=\big(z+ is(z\pm\zeta^1),\zeta^1+is(\mp z-\zeta^1)\big).
\]
Under the Hopf parametrization (we denote $\Delta_{\mathbb{S}^{2n-1}}:=\{(\alpha,\alpha):\alpha\in \mathbb{S}^{2n-1}\}$, and $z=(z_0,z')$, $\zeta=(\zeta_0,\zeta')$)
\[
\Phi:S^*\mathbb{CH}^n\simeq\big(\mathbb{S}^{2n-1}\times \mathbb{S}^{2n-1}\setminus\Delta_{\mathbb{S}^{2n-1}}\big)\times\mathbb R,
\]
\[
(z,\zeta)\longmapsto\left(\frac{z'+\zeta'}{z_0+\zeta_0},\frac{z'-\zeta'}{z_0-\zeta_0},T(z,\zeta)\right),
\]
where $T(z,\zeta)$ is a parameter on the geodesic $\varphi^{\mathbb R}(z,\zeta)$, we calculate
\[
\Phi\big(\pi_S(e^{sV^+}(z,\zeta^1,\ldots,\zeta^n))\big)
=\left(\frac{z'+(\zeta^1)'}{z_0+\zeta_0^1},\frac{z'-(\zeta^1)'+2is(z'+(\zeta^1)')}{z_0-\zeta_0^1+2is(z_0+\zeta_0^1)},T\right),
\]
\[
\Phi\big(\pi_S(e^{sV^-}(z,\zeta^1,\ldots,\zeta^n))\big)
=\left(\frac{z'+(\zeta^1)'+2is(z'-(\zeta^1)')}{z_0+\zeta_0^1+2is(z_0-\zeta_0^1)},\frac{z'-(\zeta^1)'}{z_0-\zeta_0^1},T\right).
\]
It is not hard to see that the nontrivial components
\begin{equation}
C_\pm=\left\{\frac{z'\mp(\zeta^1)'+2is(z'\pm(\zeta^1)')}{z_0\mp\zeta_0^1+2is(z_0\pm\zeta_0^1)}:s\in\mathbb R\cup\{\infty\}\right\}\label{eq:2.16}
\end{equation}
give complex circles in $\mathbb{S}^{2n-1}$. Indeed, in the notation of \eqref{eq:2.2} and \eqref{eq:2.4},
\[
z\mp\zeta^1+2is(z\pm\zeta^1)\in B_{z,\zeta^1}\cap A_0
\]
and the claim follows since $B_{z,\zeta^1}\cap\{\eta_0=1\}\cap \mathbb{S}^{2n-1}$ is a complex circle. Moreover, \eqref{eq:2.16} cannot reduce to a single point unless $z$ and $\zeta^1$ are linearly dependent, which is impossible by \eqref{eq:2.15}.

We now show that the converse holds:

\begin{lem}\label{obj:2.5}
Any complex circle $C\subset \mathbb{S}^{2n-1}$ can be written as \eqref{eq:2.16} for some $(z,\zeta)\in S^*\mathbb{CH}^n$.
\end{lem}

\begin{proof}
We will show that $C$ may be written as $C_+$ for appropriate $(z,\zeta)$. The proof for $C_-$ is analogous.

By applying an element of $SU(n,1)$ we may assume that
\[
C=\{(\zeta,0,\ldots,0):|\zeta|=1\}.
\]
This can be parametrized as
\[
C=\left\{\left(\frac{1-2is}{1+2is},0,\ldots,0\right):s\in\mathbb R\cup\{\infty\}\right\}.
\]
This is \eqref{eq:2.16} for $C_+$ with $z=(1,0,\ldots,0)$ and $\zeta=(0,-1,0,\ldots,0)$, which satisfy \eqref{eq:2.15} and hence determine an element of $S^*\mathbb{CH}^n$.
\end{proof}

Notice in particular that under the covering $S^*\mathbb{CH}^n\to S^*M$ and identification
\[
K\cap\{|\zeta|_g=1\}\simeq\big(\Lambda_\Gamma\times\Lambda_\Gamma\setminus\Delta_{\mathbb{S}^{2n-1}}\big)\times\mathbb R,
\]
where $K\subset T^*M$ is the trapped set defined in Section~\ref{subsec:modified-spectral-family}, Lemma~\ref{obj:2.5} shows that $\Lambda_\Gamma$ contains a complex circle if and only if $K\cap\{|\zeta|_g=1\}$ contains the closure of a $V^\pm$-orbit. We use this fact in Section~\ref{sec:reduction-fup} to prove the spectral gap.

\section{Scattering resolvent and fine microlocal estimates}\label{sec:scattering-resolvent}

In this section we describe the modified spectral family of the Laplacian $\mathcal P_h(\omega)$ obtained in \cite{quan2021microlocal}, which extended the method of \cite{Vas13a}, \cite{Vas13b} to the complex hyperbolic setting. We then review an anisotropic pseudodifferential calculus described in \cite[Appendix]{DyJi}, obtaining a few new results for this calculus. These two technical tools allow a straightforward adaptation of the methods of \cite{DyZa} as improved by \cite{Dy19a} to the complex hyperbolic setting and we use this to prove operator identities involving $\mathcal P_h(\omega)$ that imply fine microlocal estimates on solutions to $\mathcal P_h(\omega)u=0$. These results are combined with the Fractal Uncertainty Principle in the next section to prove the resolvent bound and spectral gap in our Theorem.

This section uses semiclassical analysis; see \cite{Zw12} and \cite[Appendix E]{DyZw} for an introduction and basic facts. Our notation will follow \cite[Section 2]{DyZa} and \cite[Section 2.3]{Dy19a}, and in particular we use
\begin{itemize}
\item the classical symbol classes $S^k(T^*M)$, $S_h^k(T^*M)$ and the corresponding class of pseudodifferential operators $\Psi_h^k(M)$;
\item the principal symbol map $\sigma_h:\Psi_h^k(M)\to S^k(T^*M)$;
\item the wavefront set $\WF_h(A)\subset\overline{T}^*M$ and the elliptic set $\operatorname{ell}_h(A)\subset\overline{T}^*M$ of $A\in\Psi_h^k(M)$ where $\overline{T}^*M$ is the fiber-radially compactified cotangent bundle;
\item the class $\Psi_h^{\rm comp}(M)\subset\bigcap_k\Psi_h^k(M)$ of compactly supported and compactly microlocalized pseudodifferential operators;
\item the class $I_h^{\rm comp}(\varkappa)$ of compactly supported and compactly microlocalized Fourier integral operators associated to an exact symplectomorphism $\varkappa:U_2\to U_1$, where $U_j\subset T^*M_j$ are open sets and $M_1,M_2$ are two manifolds of the same dimension, see, e.g., \cite[Section 2.2]{DyZa} or \cite[Section A.3]{DyJi}.
\end{itemize}
In Section~\ref{subsec:anisotropic-calculus} we will make many more definitions regarding the calculus of \cite[Appendix]{DyJi}.

\subsection{Scattering resolvent and the modified spectral family of the Laplacian.}\label{subsec:modified-spectral-family}
In this section we follow parts of \cite[Sections 4.1 and 4.2]{DyZa} and \cite[Section 2]{Dy19a} but using the modified spectral family of the Laplacian obtained in \cite{quan2021microlocal}.

Let $(M=\Gamma\backslash\mathbb{CH}^n,g)$ be a convex cocompact complex hyperbolic manifold. The geodesic flow
\[
\varphi^t:T^*M\setminus0\to T^*M\setminus0
\]
is the Hamiltonian flow of the symbol
\[
p\in C^\infty(T^*M\setminus0),\qquad p(z,\zeta)=|\zeta|_g.
\]
Just as in \cite[Section 4]{DyZa} or \cite[Section 2]{Dy19a}, we may find a function
\[
r:M\to\mathbb R,\qquad \ddot r>0\quad\text{on }\{r\geq0\}\cap\{\dot r=0\},
\]
where the dots are derivatives with respect to the flow of the lift of $r$ to $T^*M\setminus0$. Indeed, one may take $r:=\phi-r_1$ for a boundary defining function of a compactification of $M$ and a large constant $r_1>0$. Define the incoming/outgoing tails by
\[
\Gamma_\pm:=\{(z,\zeta)\in T^*M\setminus0:(\varphi^t(z,\zeta))\text{ is bounded as }t\to\mp\infty\},
\]
and the trapped set
\[
K=\Gamma_+\cap\Gamma_-.
\]
Fix $\nu_0>0$ and set
\[
\Omega:=[1-2h,1+2h]+ih[-\nu_0,1/4].
\]
Following \cite[Sections 2.3 and 2.4]{quan2021microlocal} allows us to define a semiclassical differential operator
\begin{equation}
\mathcal P_h(\omega)\in\Psi_h^2(M_{\rm ext});\qquad \mathcal P_h(\omega):=\psi_2\left(-h^2\Delta_g-\frac{h^2n^2}{4}-\omega^2\right)\psi_1\quad\text{on }M\label{eq:3.1}
\end{equation}
with $M_{\rm ext}$ a compact manifold with boundary containing $M$ as an open subset (see \cite[(2.17)]{quan2021microlocal}) and $\psi_1,\psi_2\in C^\infty(M)$ certain nonvanishing functions depending on $h,\omega$, in a way that is almost exactly analogous to the construction of Vasy \cite{Vas13a}, \cite{Vas13b} on real hyperbolic spaces. In particular, as in the paragraph before \cite[(3.14)]{Vas13b}, for any large $r_0$ we may require that
\[
\psi_1=\psi_2=1\qquad\text{near }\{r\leq r_0\},
\]
which implies that
\begin{equation}
\sigma_h(\mathcal P_h(\omega))=p^2-\omega^2\qquad\text{near }\{r\leq r_0\}.\label{eq:3.2}
\end{equation}

Some comments about the construction of the operator $\mathcal P_h(\omega)$ are in order. In the notation of \cite{quan2021microlocal}, $\mathcal P_h(\omega)$ is the operator $P_{h,\omega}-iQ_h$ where $P_{h,\omega}$ is the semiclassically rescaled, extended operator defined in \cite[Section 2.4]{quan2021microlocal}, and $Q_h\in\Psi_h^2(M_{\rm ext})$ is a complex absorbing potential supported in $M_{\rm ext}\setminus M$ satisfying the properties of \cite[Section 2.3.2]{quan2021microlocal}. This is exactly analogous to the operator used in \cite{DyZa} and \cite{Dy19a}.

A major difference, however, is that although the extended operator is still elliptic on $M$, in the complex hyperbolic setting it is no longer hyperbolic in the extended region. Instead, it is ultrahyperbolic and the characteristic set for the extended operator is connected. The source/sink estimates needed to prove the Fredholm property and hence meromorphic continuation of the extended operator may still be obtained, see \cite[Section 2.4]{quan2021microlocal}, but there is a slight complication as compared to the real hyperbolic setting. Namely, we must use variable order Sobolev spaces (see \cite[Sections 2.2.1 and 2.8]{quan2021microlocal} for definitions and relevant properties) as opposed to the standard spaces $H_h^s$, $s\in\mathbb R$. All that is really essential to our analysis is that if $s_\pm$ are real numbers such that $s_+>\frac12+\nu_0$ and $0\leq s_-<\frac14$, and $\phi(z,\zeta)\in C^\infty(T^*M_{\rm ext}\setminus0)$ is the order function constructed in \cite[Section 2.5]{quan2021microlocal}, then $\mathcal P_h(\omega):\mathcal X\to\mathcal Y$ is a Fredholm operator in a region containing $\Omega$, where
\begin{equation}
\mathcal X:=\{u\in H_h^\phi(M_{\rm ext}):\mathcal P_h(1)u\in H_h^{\phi-1}(M_{\rm ext})\},\qquad \mathcal Y:=H_h^{\phi-1}(M_{\rm ext})\label{eq:3.3}
\end{equation}
with norm
\[
\|u\|_{\mathcal X}^2:=\|u\|_{H_h^\phi(M_{\rm ext})}^2+\|\mathcal P_h(1)u\|_{H_h^{\phi-1}(M_{\rm ext})}^2.
\]
Here $H_h^\phi$ are variable order Sobolev spaces associated to the order function $\phi$, and satisfy $H_h^{s_+}\subset H_h^\phi\subset H_h^{s_-}$. The inverse of $\mathcal P_h(\omega)$ has poles of finite rank, and the set of these poles contains the set of poles of the resolvent. In particular, we remark that it follows from \eqref{eq:3.1} that for $f\in C_c^\infty(M)$, we have
\begin{equation}
\left(-h^2\Delta_g-\frac{h^2n^2}{4}-\omega^2\right)^{-1}f=\left.\psi_1\big(\mathcal P_h(\omega)^{-1}\psi_2f\big)\right|_M.\label{eq:3.4}
\end{equation}
See \cite[Section 2]{quan2021microlocal} for more on the construction and properties of $\mathcal P_h(\omega)$. Section~\ref{subsec:fine-microlocalization} shall apply the approach of \cite{DyZa} as refined in \cite{Dy19a} to further analyze the properties of this operator.

\subsection{Anisotropic semiclassical calculus and an Egorov's theorem.}\label{subsec:anisotropic-calculus}
We now briefly review the semiclassical analysis needed to prove our microlocalization results in the next section. In addition to the standard symbol classes recalled at the beginning of this section, we shall also need an anisotropic semiclassical calculus described in \cite[Appendix]{DyJi}, which generalizes a calculus originating in \cite{DyZa}. Our presentation of this calculus will follow \cite[Section 4.2.1]{ADM}.

For $M=\Gamma\backslash\mathbb{CH}^n$, a convex cocompact complex hyperbolic manifold, we let
\[
L_u:=\mathbb RX\oplus E_u,\qquad L_s:=\mathbb RX\oplus E_s
\]
be the weak unstable/stable foliations, where $X,E_u,E_s$ are defined in Section~\ref{subsec:dynamics}. Then by \cite[Lemma 2.1 and Corollary 2.3]{ADM}, $L_u,L_s$ are Lagrangian foliations; that is, every fiber of $L_u$ or $L_s$ is a Lagrangian subspace of $T(T^*M\setminus0)$, and both $L_u$ and $L_s$ are Frobenius integrable.

Fix $L\in\{L_u,L_s\}$ and two parameters
\begin{equation}
0\leq\rho<1,\qquad 0\leq\rho'\leq\frac12\rho,\qquad \rho+\rho'<1.\label{eq:3.5}
\end{equation}
Then we define the symbol class $S^{\rm comp}_{L,\rho,\rho'}(T^*M\setminus0)$ from \cite[Section A.1]{DyJi} to be the smooth, $h$-dependent functions $a(z,\zeta;h)$ with support contained in an $h$-independent compact set in $T^*M\setminus0$, and satisfying the following derivative bounds: For any vector fields $Y_1,\ldots,Y_m$, $Q_1,\ldots,Q_k$ on $T^*M\setminus0$ with the $Y_i$, $i=1,\ldots,m$ tangent to $L$, there exists a constant $C$ independent of $h$ such that
\[
\sup_{z,\zeta}|Y_1\cdots Y_mQ_1\cdots Q_ka(z,\zeta;h)|\leq Ch^{-\rho k-\rho'm},\qquad 0<h\leq1.
\]
Such symbols may be quantized using the quantization $\Op_h^L$ defined in \cite[Section A.4]{DyJi}. Indeed, if $S^{\rm comp}_{L_0,\rho,\rho'}(T^*\mathbb R^{2n})$ is the model symbol class constructed in \cite[Section A.2]{DyJi}, and $a\in S^{\rm comp}_{L,\rho,\rho'}(T^*M\setminus0)$, then there are open sets $U_\ell\subset T^*M\setminus0$ forming a locally finite cover of $\supp a$, and we may write
\[
\Op_h^L(a):=\sum_{\ell=1}^N B_\ell'\Op_h(a_\ell)B_\ell,\qquad a_\ell=(\chi_\ell a)\circ\varkappa_\ell^{-1}\in S^{\rm comp}_{L_0,\rho,\rho'}(T^*\mathbb R^{2n}),
\]
where for each $\ell$, $\varkappa_\ell:U_\ell\to T^*\mathbb R^{2n}$ is an exact symplectomorphism onto its image that maps $L$ to $L_0$, $B_\ell\in I_h^{\rm comp}(\varkappa_\ell)$, $B_\ell'\in I_h^{\rm comp}(\varkappa_\ell^{-1})$, the $\sigma_h(B_\ell'B_\ell)\in C_c^\infty(U_\ell)$ form a partition of unity on $\supp a$, $\chi_\ell\in C_c^\infty(U_\ell)$ is equal to $1$ near $\supp\sigma_h(B_\ell'B_\ell)$, and $\Op_h(a_\ell)$ is the standard quantization \cite[(A.5)]{DyJi}:
\[
\Op_h(a_\ell)f(y)=(2\pi h)^{-2n}\int_{\mathbb R^{4n}}e^{\frac{i}{h}(y-y')\cdot\eta}a_\ell(y,\eta)f(y')\,dy'\,d\eta.
\]

Similarly, a family $A=A(h):\mathcal D'(M)\to C_c^\infty(M)$ is in the class $\Psi^{\rm comp}_{h,L,\rho,\rho'}(M)$ if there exist open sets $U_\ell\subset T^*M\setminus0$ such that $A$ can be written
\begin{equation}
A=\sum_{\ell=1}^N B_\ell'\Op_h(a_\ell)B_\ell+O(h^\infty)_{\mathcal D'(M)\to C_c^\infty(M)}\label{eq:3.6}
\end{equation}
for $\varkappa_\ell$ and $B_\ell\in I_h^{\rm comp}(\varkappa_\ell)$, $B_\ell'\in I_h^{\rm comp}(\varkappa_\ell^{-1})$ as above, and some $a_\ell\in S^{\rm comp}_{L_0,\rho,\rho'}(T^*\mathbb R^{2n})$. In this case, we define the principal symbol $\sigma_h^L(A)\in S^{\rm comp}_{L,\rho,\rho'}(T^*M\setminus0)/h^{1-\rho-\rho'}S^{\rm comp}_{L,\rho,\rho'}(T^*M\setminus0)$ of $A$ by
\[
\sigma_h^L(A):=\sum_{\ell=1}^N\sigma_h(B_\ell'B_\ell)(a_\ell\circ\varkappa_\ell),
\]
which is valid for any representation \eqref{eq:3.6}.

We refer to \cite[Appendix]{DyJi} and \cite[Section 4.2.1]{ADM} for basic properties of this calculus that we shall use throughout the rest of the paper. Here we recall a few additional properties needed in the next section. These properties are analogues of results in \cite[Section 3.3]{DyZa} and some of them have similar versions stated in \cite[Appendix]{DyJi}. However, with minor alterations and minimal work in this section, we give versions that are exactly analogous to the statements in \cite[Section 3.3]{DyZa}. This allows us to simply quote these analogues to prove the more complicated results in the next section, which generalize \cite[Lemmas 2.2--2.7]{DyZa}. In other words, it is easier for us to modify the semiclassical statements in this section than to use the versions of \cite[Appendix]{DyJi} and modify the proofs of results in the next section, though either approach is possible.

We begin with the following definition, which is \cite[Definition 3.4]{DyZa}:

\begin{defi}\label{obj:3.1}
Let $a\in S^{\rm comp}_{L,\rho,\rho'}(T^*M\setminus0)$ and let $h_j\to0$, $(z_j,\zeta_j)\in T^*M\setminus0$ be some sequences. We say that $a$ is $O(h^\infty)$ along $(z_j,\zeta_j;h_j)$ if for each $N$ and any vector fields $Z_1,\ldots,Z_m$ on $T^*M\setminus0$, there exists a constant $C$ such that
\[
|Z_1\cdots Z_ma(z_j,\zeta_j;h_j)|\leq Ch_j^N.
\]
\end{defi}

We now prove the following analogue of \cite[Lemma 3.12]{DyZa}. To state it, we first recall the following definition from \cite[Definition E.36]{DyZw}. Let $M_1,M_2$ be manifolds and
\[
B=B(h):C_c^\infty(M_2)\to\mathcal D'(M_1)
\]
an $h$-tempered family of operators. Then $\WF'_h(B)\subset\overline{T^*(M_1\times M_2)}$ is defined as
\[
\WF'_h(B):=\{(z,\zeta,y,\eta):(z,\zeta,y,-\eta)\in\WF_h(K_B)\}
\]
where $K_B(z,y)\in\mathcal D'(M_1\times M_2)$ is the Schwartz kernel of $B$. We let $\pi_1,\pi_2$ denote the projections onto the first two and second two variables, respectively.

\begin{lem}\label{obj:3.2}
Let $A\in\Psi^{\rm comp}_{h,L,\rho,\rho'}(M)$. Then
\begin{enumerate}[1.]
\item $A$ is bounded on $L^2(M)$ uniformly in $h$ and $\WF_h(A)$ is compact.
\item If $U$ is an open set containing $\WF_h(A)$, $\varkappa:\overline U\to T^*\mathbb R^{2n}$ is an exact symplectomorphism onto its image which maps $L$ to $L_0$, and $B\in I_h^{\rm comp}(\varkappa)$, $B'\in I_h^{\rm comp}(\varkappa^{-1})$, then
\[
BAB'=\Op_h(a)+O(h^\infty)_{L^2\to L^2}
\]
for some $a\in S^{\rm comp}_{L_0,\rho,\rho'}(T^*\mathbb R^{2n})$, $\supp a\subset\varkappa(\overline U)$, and for each representation \eqref{eq:3.6} of $A$,
\begin{equation}
a\circ\varkappa=\sigma_h(B'B)\sum_{\ell=1}^N\sigma_h(B_\ell'B_\ell)(a_\ell\circ\varkappa_\ell)+O(h^{1-\rho})_{S^{\rm comp}_{L_0,\rho,\rho'}}.\label{eq:3.7}
\end{equation}
Moreover, if $h_j\to0$ and $(z_j,\zeta_j)\in U$ are sequences such that for each $\ell$, either $(z_j,\zeta_j)\notin\pi_1(\WF'_h(B_\ell))\cap\pi_2(\WF'_h(B_\ell'))$ for all $j$ or $a_\ell\circ\varkappa_\ell$ is $O(h^\infty)$ along $(z_j,\zeta_j;h_j)$ for all $\ell$ in the sense of Definition~\ref{obj:3.1}, then $a\circ\varkappa$ is $O(h^\infty)$ along $(z_j,\zeta_j;h_j)$ as well.
\end{enumerate}
\end{lem}

\begin{proof}
We follow the proof of \cite[Lemma 3.12]{DyZa}. Part one follows easily from properties of the symbol class $S^{\rm comp}_{L_0,\rho,\rho'}(T^*\mathbb R^{2n})$, see \cite[Section A.2]{DyJi}, and the representation \eqref{eq:3.6}.

Next write $A$ using \eqref{eq:3.6} to find
\[
BAB'=\sum_{\ell=1}^N(BB_\ell')\Op_h(a_\ell)(B_\ell B')+O(h^\infty)_{\mathcal D'\to C_c^\infty}.
\]
If $U_\ell$ is the domain of $\varkappa_\ell$, the function $\varkappa_\ell':=\varkappa_\ell\circ\varkappa^{-1}:\varkappa(U\cap U_\ell)\to T^*\mathbb R^{2n}$ is a symplectomorphism onto its image and preserves $L_0$. Thus $B_\ell B'\in I_h^{\rm comp}(\varkappa_\ell')$ and $B B_\ell'\in I_h^{\rm comp}((\varkappa_\ell')^{-1})$. We conclude part two by applying \cite[Lemma A.3]{DyJi} (which adapts \cite[Lemma 3.10]{DyZa} to the class $S^{\rm comp}_{L_0,\rho,\rho'}$). Then \eqref{eq:3.7} follows from the identity $\sigma_h(BB_\ell' B_\ell B')=(\sigma_h(B_\ell'B_\ell)\sigma_h(B'B))\circ\varkappa^{-1}$, which in turn follows from \cite[(2.12)]{DyZa}.

The final statement is obtained as follows: From the proof of \cite[Lemma A.3]{DyJi} (which is an analogue of \cite[Lemma 3.10]{DyZa}), we obtain a microlocal vanishing statement identical to the end of the latter lemma. Then the final statement of Lemma~\ref{obj:3.2} follows in the same way that the final statement of \cite[Lemma 3.12]{DyZa} follows from \cite[Lemma 3.10]{DyZa}. We conclude the proof.
\end{proof}

Next we improve the bound stated in part 1 of Lemma~\ref{obj:3.2}. It follows in exactly the same way as in \cite[Lemma 3.15]{DyZa}, see also \cite[Lemma A.5]{DyJi}, which makes no use of the assumption that $M$ be compact.

\begin{prop}\label{obj:3.3}
Let $A\in\Psi^{\rm comp}_{h,L,\rho,\rho'}(M)$. Then as $h\to0$,
\[
\|A\|_{L^2(M)\to L^2(M)}\leq\sup|\sigma_h^L(A)|+o(1).
\]
\end{prop}

We now use Lemma~\ref{obj:3.2} to make the following definition, analogous to \cite[Definition 3.13]{DyZa}:

\begin{defi}\label{obj:3.4}
Let $A\in\Psi^{\rm comp}_{h,L,\rho,\rho'}(M)$. If $h_j\to0$ and $(z_j,\zeta_j)\in T^*M\setminus0$ are sequences, then we say that $A=O(h^\infty)$ microlocally along $(z_j,\zeta_j;h_j)$ if for each choice of $\varkappa,B,B'$ in Part 2 of Lemma~\ref{obj:3.2} and the corresponding symbol $a\in S^{\rm comp}_{L_0,\rho,\rho'}(T^*\mathbb R^{2n})$, the symbol $a\circ\varkappa$ is $O(h^\infty)$ along $(z_j,\zeta_j;h_j)$.
\end{defi}

The following is a direct consequence of Part 2 of Lemma~\ref{obj:3.2}, and is an analogue of Part 4 of \cite[Lemma 3.14]{DyZa}:

\begin{prop}\label{obj:3.5}
For each $a\in S^{\rm comp}_{L,\rho,\rho'}(T^*M\setminus0)$, if $h_j\to0$ and $(z_j,\zeta_j)\in T^*M\setminus0$ are sequences such that $a$ is $O(h^\infty)$ along $(z_j,\zeta_j;h_j)$, then $\Op_h^L(a)$ is $O(h^\infty)$ microlocally along $(z_j,\zeta_j;h_j)$.
\end{prop}

Proposition~\ref{obj:3.5} allows a direct analogue of the version of an elliptic parametrix construction in \cite[Lemma 3.16]{DyZa}; the proof follows as in the proof of \cite[Lemma 3.16]{DyZa}, using Proposition~\ref{obj:3.5} in place of Part 4 of \cite[Lemma 3.14]{DyZa}.

\begin{prop}\label{obj:3.6}
Assume that $A,B\in\Psi^{\rm comp}_{h,L,\rho,\rho'}(M)$ and $B$ is elliptic on the microsupport of $A$ in the following sense: there exists $\varepsilon>0$ such that for any sequences $h_j\to0$ and $(z_j,\zeta_j)\in T^*M\setminus0$, if $|\sigma_h^L(B)(z_j,\zeta_j;h_j)|\leq\varepsilon$, then $A$ is $O(h^\infty)$ microlocally along $(z_j,\zeta_j;h_j)$. Then there exists
\[
Q\in\Psi^{\rm comp}_{h,L,\rho,\rho'}(M),\qquad A=QB+O(h^\infty)_{\mathcal D'\to C_c^\infty}.
\]
\end{prop}

For the last result of this section, we will prove an Egorov's theorem for the $\Psi^{\rm comp}_{h,L,\rho,\rho'}(M)$ calculus that is a direct analogue of \cite[Lemma 3.17]{DyZa}. We begin by recalling from \cite[Section A.2]{DyJi} that for $a,b\in S^{\rm comp}_{L_0,\rho,\rho'}(T^*\mathbb R^{2n})$, there exists a symbol $a\# b\in S^{\rm comp}_{L_0,\rho,\rho'}(T^*\mathbb R^{2n})$ such that
\[
\Op_h(a)\Op_h(b)=\Op_h(a\# b)+O(h^\infty)_{L^2\to L^2}.
\]
Moreover, it can be shown via rescaling and applying \cite[Theorems 4.14 and 4.17]{Zw12} as in the proofs of \cite[Lemma 3.8]{DyZa} that for each $N$,
\begin{equation}
a\# b(y,\eta;h)=\sum_{j=0}^{N-1}\frac{(-ih)^j}{j!}(\partial_\eta\cdot\partial_{y'})^j\big(a(y,\eta;h)b(y',\eta';h)\big)\big|_{y'=y,\eta'=\eta}+O(h^{(1-\rho-\rho')N})_{S^{\rm comp}_{L_0,\rho,\rho'}}.\label{eq:3.8}
\end{equation}
We now state the following, similarly to \cite[Lemma 3.17]{DyZa}:

\begin{lem}\label{obj:3.7}
Fix $L\in\{L_u,L_s\}$, $P\in\Psi_h^{\rm comp}(M)$ with real-valued $h$-independent principal symbol $p:=\sigma_h(P)$ satisfying
\begin{equation}
L_{(z,\zeta)}\subset\ker dp(z,\zeta),\qquad (z,\zeta)\in T^*M\setminus0.\label{eq:3.9}
\end{equation}
Let $A\in\Psi^{\rm comp}_{h,L,\rho,\rho'}(M)$ and let $U$ be an open set containing $\WF_h(A)$. Take $T>0$ such that $e^{-tH_p}(\WF_h(A))\subset U$ for all $t\in[0,T]$ where $e^{tH_p}$ is the Hamiltonian flow of $p$. Then there exists a family of operators depending smoothly on $t$
\[
A_t\in\Psi^{\rm comp}_{h,L,\rho,\rho'}(M),\quad t\in[0,T],\qquad A_0=A+O(h^\infty)_{\mathcal D'\to C_c^\infty},
\]
such that
\begin{equation*}
\sigma_h^L(A_t)=\sigma_h^L(A)\circ e^{tH_p}+O(h^{1-\rho-\rho'})_{S^{\rm comp}_{L,\rho,\rho'}(T^*M\setminus0)}
\end{equation*}
and
\begin{equation}
    \label{eq:3.10}
    ih\partial_tA_t+[P,A_t]=O(h^\infty)_{\mathcal D'\to C_c^\infty}.
\end{equation}
Moreover, if $t$ is fixed and $h_j\to0$, $(z_j,\zeta_j)\in T^*M\setminus0$ are sequences such that $A$ is $O(h^\infty)$ microlocally along $(z_j,\zeta_j;h_j)$, then $A_t$ is $O(h^\infty)$ microlocally along $(e^{-tH_p}(z_j,\zeta_j);h_j)$.
\end{lem}

\begin{proof}
We follow the proof of \cite[Lemma 3.17]{DyZa}, using our modifications obtained above. By \eqref{eq:3.9}, the Hamiltonian vector field $H_p$ is in $L$, and thus $e^{tH_p}$ preserves $L$. It follows that for any $a\in S^{\rm comp}_{L,\rho,\rho'}(T^*M\setminus0)$ both $H_pa$ and $a\circ e^{tH_p}$ are in $S^{\rm comp}_{L,\rho,\rho'}(T^*M\setminus0)$ as well. We claim that for any $\widetilde A\in\Psi^{\rm comp}_{h,L,\rho,\rho'}(M)$,
\begin{equation}
[P,\widetilde A]=\frac{h}{i}\Op_h^L\big(H_p\sigma_h^L(\widetilde A)\big)+O(h^{2-\rho-\rho'})_{\Psi^{\rm comp}_{h,L,\rho,\rho'}(M)}.\label{eq:3.11}
\end{equation}
To prove \eqref{eq:3.11}, note that it follows from \eqref{eq:3.8} that if $f\in C_c^\infty(\mathbb R^{2n})$ and $a\in S^{\rm comp}_{L_0,\rho,\rho'}(T^*\mathbb R^{2n})$, then
\begin{equation}
f(y)\# a-a\# f(y)=\frac{h}{i}\{f(y),a\}+O(h^2)_{S^{\rm comp}_{L_0,\rho,\rho'}(T^*\mathbb R^{2n})}.\label{eq:3.12}
\end{equation}
As in the proof of \cite[Lemma 3.17]{DyZa}, applying \eqref{eq:3.12}, a pseudodifferential partition of unity, and Part 2 of Lemma~\ref{obj:3.2}, we obtain \eqref{eq:3.11}.

Having proved \eqref{eq:3.11}, the remainder of the proof follows as in the proof of \cite[Lemma 3.17]{DyZa}.
\end{proof}

We remark that another version of Egorov's theorem for the symbol class $S^{\rm comp}_{L,\rho,\rho'}(T^*M\setminus0)$ was proved in \cite[Lemma A.7]{DyJi}.

\subsection{Fine microlocalization identities for the modified Laplacian $\mathcal P_h(\omega)$.}\label{subsec:fine-microlocalization}
Our work in the prior section makes it possible to obtain operator identities for $\mathcal P_h(\omega)$ similar to the ones obtained by \cite{Dy19a} for the extended operator of Vasy in the real hyperbolic setting. Lemmas~\ref{obj:3.8}--\ref{obj:3.11} are direct analogues of \cite[Lemmas 2.2--2.5]{Dy19a} and are obtained with ease using our analogues of \cite[Lemmas 3.15--3.17]{DyZa} obtained in Section~\ref{subsec:anisotropic-calculus}. However, the main results of this section, namely Lemmas~\ref{obj:3.12} and \ref{obj:3.13}, are significant modifications of the refined microlocal identities of \cite[Lemmas 2.6 and 2.7]{Dy19a}. These lemmas give a sharper localization to the trapped set than the corresponding results of \cite{Dy19a}, and are the crucial tool allowing an adaptation of the approach of \cite{ADM} in the next section.

The first two results are similar to the statements of \cite[Lemmas 2.2 and 2.3]{Dy19a}, but rather than the extended operator of Vasy, we are using the operator of Quan \cite{quan2021microlocal} recalled in Section~\ref{subsec:modified-spectral-family}. These lemmas (Lemmas~\ref{obj:3.8} and \ref{obj:3.9} below) allow one to treat the infinity of $M$ as a black box and to focus our analysis near the trapped set. As with \cite[Lemma 2.2]{Dy19a}, our Lemma~\ref{obj:3.8} below says that solutions to $\mathcal P_h(\omega)u=0$ (called resonant states), when restricted to $\{r\leq r_0\}$, are microlocally negligible outside $h$-independent neighborhoods of $\Gamma_+\cap\{|\zeta|_g=1\}$. Our notation is from Section~\ref{subsec:modified-spectral-family}; in particular $\nu_0>0$ and hence $\Omega$ are fixed, and then $\mathcal X,\mathcal Y$ are the spaces defined in \eqref{eq:3.3}.

\begin{lem}\label{obj:3.8}
Assume that $A_1\in\Psi_h^0(M_{\rm ext})$, $\WF_h(A_1)\subset\{r\leq r_0\}\subset\overline{T}^*M$, and
\begin{equation}
\WF_h(A_1)\cap\Gamma_+\cap\{|\zeta|_g=1\}=\varnothing.\label{eq:3.13}
\end{equation}
Then we have on $\mathcal X$,
\begin{equation}
A_1=Z_1(\omega)\mathcal P_h(\omega)+O(h^\infty)_{\mathcal D'\to\mathcal X},\label{eq:3.14}
\end{equation}
where $Z_1(\omega)$ is holomorphic in $\omega\in\Omega$ and $\|Z_1(\omega)\|_{\mathcal Y\to\mathcal X}\leq Ch^{-1}$.
\end{lem}

\begin{proof}
We follow the proof of \cite[Lemma 2.2]{Dy19a}, and as in that proof we may use \eqref{eq:3.13} to find a complex absorbing potential
\[
Q_1\in\Psi_h^{\rm comp}(M),\qquad \sigma_h(Q_1)\geq0,
\]
\[
K\cap\{|\zeta|_g=1\}\subset\operatorname{ell}_h(Q_1),\qquad \WF_h(Q_1)\subset\{r\leq r_0\},
\]
such that
\[
\WF_h(Q_1)\cap\bigcup_{t\geq0}\varphi^{-t}\big(\WF_h(A_1)\cap\{|\zeta|_g=1\}\big)=\varnothing.
\]
Now, if $Q_h\in\Psi_h^2(M_{\rm ext})$ is the complex absorbing potential that is part of $\mathcal P_h(\omega)$ (see Section~\ref{subsec:modified-spectral-family}), then $Q_1':=Q_h+Q_1$ satisfies the properties of paragraph two of \cite[Section 2.4.1]{quan2021microlocal}. Thus, $\mathcal P_h(\omega)-iQ_1:\mathcal X\to\mathcal Y$ is semiclassically outgoing in the sense of \cite[Section 2.4.1]{quan2021microlocal} and we may apply \cite[Theorem 2.5]{quan2021microlocal} to see that 
\[\mathcal P_h(\omega)-iQ_1:\mathcal X\to\mathcal Y\]
is invertible for small $h$ and its inverse satisfies the high-energy bound
\[
\|(\mathcal P_h(\omega)-iQ_1)^{-1}\|_{\mathcal Y\to\mathcal X}\leq Ch^{-1}.
\]
Moreover, the semiclassical outgoing property and the properties of $A_1,Q_1$ imply that
\begin{equation}
A_1(\mathcal P_h(\omega)-iQ_1)^{-1}Q_1=O(h^\infty)_{\mathcal D'\to\mathcal X}.\label{eq:3.15}
\end{equation}
Here we use that $Q_1$ is uniformly bounded in $h$ as an operator $H_h^{-N}(M_{\rm ext})\to\mathcal Y$ for all $N$. Setting
\[
Z_1(\omega):=A_1(\mathcal P_h(\omega)-iQ_1)^{-1},
\]
the statement follows, as
\[
A_1-Z_1(\omega)\mathcal P_h(\omega)=-iA_1(\mathcal P_h(\omega)-iQ_1)^{-1}Q_1.
\]
\end{proof}

\begin{rem}
Note that the remainder in \eqref{eq:3.14} is $\mathcal D'\to\mathcal X$ rather than the $\mathcal D'\to C^\infty$ bound obtained in the version of \cite[Lemma 2.2]{Dy19a}. This is related to the fact that the variable order Sobolev spaces satisfy $H_h^{s_+}(M_{\rm ext})\subset H_h^\phi(M_{\rm ext})\subset H_h^{s_-}(M_{\rm ext})$, where $s_+$ may be taken arbitrarily large but $0\leq s_-<\frac14$, and thus we cannot obtain the improved regularity bound in \eqref{eq:3.15} as was done in the proof of \cite[Lemma 2.2]{Dy19a}.
\end{rem}

Our next result is an analogue of \cite[Lemma 2.3]{Dy19a}, and implies that resonant states are completely determined by their microlocal behavior in any $h$-independent neighborhood of $K\cap\{|\zeta|_g=1\}$.

\begin{lem}\label{obj:3.9}
Assume that $A_2\in\Psi_h^0(M_{\rm ext})$ is elliptic on $K\cap\{|\zeta|_g=1\}$. Then on $\mathcal X$,
\begin{equation}
I=Z_2(\omega)\mathcal P_h(\omega)+J_2(\omega)A_2+O(h^\infty)_{\mathcal D'\to\mathcal X},\label{eq:3.16}
\end{equation}
where $Z_2(\omega),J_2(\omega)$ are holomorphic in $\omega\in\Omega$ and satisfy $\|Z_2(\omega)\|_{\mathcal Y\to\mathcal X}\leq Ch^{-1}$, $\|J_2(\omega)\|_{H_h^{-N}\to\mathcal X}\leq C_N$ for all $N$.
\end{lem}

The proof is identical to the proof of \cite[Lemma 2.3]{Dy19a}, with the changes in the remainder bound in \eqref{eq:3.16}, as well as in the bound on $J_2(\omega)$, coming from the issue in the Remark after Lemma~\ref{obj:3.8}. As mentioned above, Lemmas~\ref{obj:3.8} and \ref{obj:3.9} allow us to treat infinity as a black box and focus on the analysis near the trapped set. In particular, using also \eqref{eq:2.9} and \eqref{eq:2.10}, the hypotheses of \cite{DyGu} are satisfied in our setting. This observation is applied in the proofs of Lemmas~\ref{obj:3.12} and \ref{obj:3.13} below to obtain further properties of the flow (see \eqref{eq:3.19} and \eqref{eq:3.20}).

We now give an approximate inverse statement for the class $\Psi^{\rm comp}_{h,L,\rho,\rho'}(M)$, $L\in\{L_u,L_s\}$ reviewed in Section~\ref{subsec:anisotropic-calculus}. It is an analogue of \cite[Lemma 2.4]{Dy19a}, which gives a similar statement for the class $\Psi^{\rm comp}_{h,L,\rho}(M)$, real hyperbolic manifolds, and the extended operator of Vasy.

\begin{lem}\label{obj:3.10}
Let $a,b\in S^{\rm comp}_{L,\rho,\rho'}(T^*M\setminus0)$ where $L\in\{L_u,L_s\}$, $\rho,\rho'$ satisfy \eqref{eq:3.5}, and fix $T>0$. Assume that $|a|\leq1$ everywhere and
\begin{equation}
\varphi^{-T}(\supp a)\subset\{b=1\},\qquad \varphi^{-t}(\supp a)\subset W_0,\quad t\in[0,T],\label{eq:3.17}
\end{equation}
where $W_0:=\{r\leq r_0\}\cap\{\frac12\leq|\zeta|_g\leq2\}\subset T^*M\setminus0$. Then
\[
\Op_h^L(a)=Z(\omega)\mathcal P_h(\omega)+J(\omega)\Op_h^L(b)+O(h^\infty)_{\mathcal D'\to C^\infty}
\]
where $Z(\omega),J(\omega):\mathcal D'(M_{\rm ext})\to C^\infty(M_{\rm ext})$ are holomorphic in $\omega\in\Omega$ and for all $N$,
\[
\|Z(\omega)\|_{H_h^{-N}\to H_h^N}\leq C_Nh^{-1},\qquad \|J(\omega)\|_{H_h^{-N}\to H_h^N}\leq C_N,
\]
and for each $\varepsilon_1>0$ and $h$ small enough depending on $\varepsilon_1$,
\[
\|J(\omega)\|_{L^2\to L^2}\leq\exp(-T\Im\omega/h)+\varepsilon_1.
\]
\end{lem}

\begin{proof}
We follow the proof of \cite[Lemma 2.4]{Dy19a}, using our Proposition~\ref{obj:3.3}, Proposition~\ref{obj:3.6}, and Lemma~\ref{obj:3.7} in place of \cite[Lemmas 3.15--3.17]{DyZa}. All that needs to be shown is that we may construct a self-adjoint operator $P\in\Psi_h^{\rm comp}(M)$ with
\[
P^2=-h^2\Delta_g-\frac{h^2n^2}{4}+O(h^\infty)\qquad\text{microlocally near }W_0,
\]
\[
\sigma_h(P)=p=|\zeta|_g\qquad\text{near }W_0.
\]
But this follows exactly as in \cite[(4.22)]{DyZa}, which in turn follows \cite[Lemma 4.6]{GrSj} and we conclude the proof.
\end{proof}

The following analogue of \cite[Lemma 2.5]{Dy19a} gives a version of Lemma~\ref{obj:3.8} for the class $\Psi^{\rm comp}_{h,L,\rho,\rho'}(M)$.

\begin{lem}\label{obj:3.11}
Let $a\in S^{\rm comp}_{L,\rho,\rho'}(T^*M\setminus0)$ where $L\in\{L_u,L_s\}$, $\rho,\rho'$ satisfy \eqref{eq:3.5}, and assume that $\supp a\subset V$ where
\[
V\subset\{r<r_0\}\setminus(\Gamma_+\cap\{|\zeta|_g=1\})
\]
is an $h$-independent compact subset. Then we have on $\mathcal X$
\[
\Op_h^L(a)=Z(\omega)\mathcal P_h(\omega)+O(h^\infty)_{\mathcal D'\to\mathcal X}
\]
where $Z(\omega)$ is holomorphic in $\omega\in\Omega$ and $\|Z(\omega)\|_{\mathcal Y\to\mathcal X}\leq Ch^{-1}$.
\end{lem}

The proof is the same as the proof of \cite[Lemma 2.5]{Dy19a}, using Proposition~\ref{obj:3.6} and Lemma~\ref{obj:3.8} in place of \cite[Lemma 3.16]{DyZa} and \cite[Lemma 2.2]{Dy19a}, respectively.

To this point, everything has been more-or-less straightforward adaptations of \cite[Section 2]{Dy19a} to the class $\Psi^{\rm comp}_{h,L,\rho,\rho'}(M)$, the flow $\varphi^t$, and the operator $\mathcal P_h(\omega)$. Starting here, we make the first significant modification. We give versions of \cite[Lemmas 2.6 and 2.7]{Dy19a} but for symbols closely related to those studied in \cite[Section 5]{ADM}. To prove the spectral gap in the next section, we show that in the absence of complex circles in the limit set, such symbols have support projecting to porous sets in the fast direction, and the application of the Fractal Uncertainty Principle via the method of \cite[Section 5]{ADM} then concludes the proof.

The following lemma is analogous to \cite[Lemma 2.6]{Dy19a}, and refines Lemma~\ref{obj:3.8}.

\begin{lem}\label{obj:3.12}
Fix $\rho\in[0,1)$, and $\chi\in C_c^\infty(T^*M\setminus0;[0,1])$ with $\chi=1$ near $K\cap\{|\zeta|_g=1\}$. Then there exists an integer $n_0>0$ such that uniformly in $\omega\in\Omega$,
\[
\Op_h^{L_u}\left(\chi\left(1-\prod_{m=1}^k\chi\circ\varphi^{-mn_0}\right)\right)=Z_+(\omega)\mathcal P_h(\omega)+O(h^\infty)_{\mathcal D'\to\mathcal X}
\]
where $k=k(h)$ is the smallest integer such that $\frac\rho2\log\frac1h\leq kn_0$, and $Z_+(\omega)$ is holomorphic in $\omega\in\Omega$ and satisfies
\begin{equation}
\|Z_+(\omega)\|_{\mathcal Y\to\mathcal X}\leq Ch^{-1-\rho \nu_0}.\label{eq:3.18}
\end{equation}
\end{lem}

\begin{rem}
Our choice of $n_0$ and $k(h)$ implies that $kn_0\sim\frac\rho2\log\frac1h$. Thus in particular, the symbol $\chi\left(1-\prod_{m=1}^k\chi\circ\varphi^{-mn_0}\right)\in S^{\rm comp}_{L_u,\rho+\varepsilon,\varepsilon}$ for small $\varepsilon$, as it is a product of $k(h)\leq C\log\frac1h$ symbols in $S^{\rm comp}_{L_u,\rho,0}$ by \cite[Lemma 4.1]{ADM}.
\end{rem}

\begin{proof}
We closely follow the proof of \cite[Lemma 2.6]{Dy19a}, but with modifications allowing our improvement. Fix $T_0>0$ so that for all $(z,\zeta)\in\{|\zeta|_g=1\}$ and $t,t_1,t_2\geq T_0$, we have
\begin{equation}
(z,\zeta)\in\Gamma_+\cap\supp\chi\quad\Longrightarrow\quad\varphi^{-t}(z,\zeta)\notin\supp(1-\chi),\label{eq:3.19}
\end{equation}
\begin{equation}
(z,\zeta)\in\varphi^{t_1}(\supp\chi)\cap\varphi^{-t_2}(\supp\chi)\quad\Longrightarrow\quad(z,\zeta)\notin\supp(1-\chi).\label{eq:3.20}
\end{equation}
Such $T_0$ exists by \cite[Lemmas 2.3 and 2.4]{DyGu}, using that $\chi=1$ near $K\cap\{|\zeta|_g=1\}$. Choose an integer $n_0\in [2T_0,4T_0]$. Note that there exists $\varepsilon_1>0$ such that
\begin{equation}
\exp(-(n_0+T_0)\Im\omega/h)+\varepsilon_1\leq\exp(n_0(\nu_0-\Im\omega/h)).\label{eq:3.21}
\end{equation}
Define
\[
A_+^j:=\Op_h^{L_u}\left(\chi\left(1-\prod_{m=1}^j\chi\circ\varphi^{-mn_0}\right)\right).
\]
With $k=k(h)$ as in the statement of the lemma, we claim that uniformly in $j=1,\ldots,k-1$,
\begin{equation}
A_+^{j+1}=Z_+^j(\omega)\mathcal P_h(\omega)+J_+^j(\omega)A_+^j+O(h^\infty)_{\mathcal D'\to\mathcal X},\label{eq:3.22}
\end{equation}
where $Z_+^j(\omega),J_+^j(\omega)$ are holomorphic in $\omega\in\Omega$ and for all $N$,
\begin{equation}
\|Z_+^j(\omega)\|_{\mathcal Y\to\mathcal X}\leq Ch^{-1},\qquad
\|J_+^j(\omega)\|_{H_h^{-N}\to H_h^N}\leq C_N,
\qquad
\|J_+^j(\omega)\|_{L^2\to L^2}\leq\exp(n_0(\nu_0-\Im\omega/h)).\label{eq:3.23}
\end{equation}
In order to prove this, we decompose $\chi=\chi_1+\chi_2$ with $\chi_1,\chi_2\in C_c^\infty(T^*M\setminus0;[0,1])$,
\[
\supp\chi_1\subset\left\{\frac12\leq|\zeta|_g\leq2\right\},\qquad \supp\chi_2\cap\Gamma_+\cap\{|\zeta|_g=1\}=\varnothing,
\]
where $\chi_1,\chi_2$ are independent of $j,h$, and for all $t\in[T_0,5T_0]$, $t_1,t_2\geq T_0$, we have
\begin{equation}
(z,\zeta)\in\supp\chi_1\quad\Longrightarrow\quad\varphi^{-t}(z,\zeta)\notin\supp(1-\chi),\label{eq:3.24}
\end{equation}
\begin{equation}
(z,\zeta)\in\varphi^{t_1}(\supp\chi)\cap\varphi^{-t_2}(\supp\chi_1)\quad\Longrightarrow\quad(z,\zeta)\notin\supp(1-\chi),\label{eq:3.25}
\end{equation}
\begin{equation}
(z,\zeta)\in\varphi^{t_1}(\supp\chi_1)\cap\varphi^{-t_2}(\supp\chi)\quad\Longrightarrow\quad(z,\zeta)\notin\supp(1-\chi).\label{eq:3.26}
\end{equation}
(Note that \eqref{eq:3.24} is easy to arrange, and that \eqref{eq:3.25},\eqref{eq:3.26} follow automatically from \eqref{eq:3.20} as long as $\supp\chi_1\subset\supp\chi$.) We claim that for each $j=1,\ldots,k-1$,
\begin{equation}
\varphi^{-(n_0+T_0)}\left(\supp\left(\chi_1\left(1-\prod_{m=1}^{j+1}\chi\circ\varphi^{-mn_0}\right)\right)\right)
\subset\left\{\chi\left(1-\prod_{m=1}^j\chi\circ\varphi^{-mn_0}\right)=1\right\}.\label{eq:3.27}
\end{equation}
Indeed, let $(z,\zeta)\in\supp\left(\chi_1\left(1-\prod_{m=1}^{j+1}\chi\circ\varphi^{-mn_0}\right)\right)$. We need to show that
\begin{enumerate}[i)]
\item $\chi(\varphi^{-(n_0+T_0)}(z,\zeta))=1$,
\item $\chi(\varphi^{-(m+1)n_0-T_0}(z,\zeta))=0$, for some $m \in \{1,\ldots,j\}$.
\end{enumerate}
Now, (i) follows from \eqref{eq:3.24} since $n_0+T_0\in[T_0,5T_0]$. To prove (ii), we note that $\varphi^{-rn_0}(z,\zeta)\in\supp(1-\chi)$ for some $r \in\{1,\ldots,j+1  \}$. Applying \eqref{eq:3.25} to $\varphi^{-rn_0}(z,\zeta)\in\supp(1-\chi)$ with $t_1=T_0$, $t_2=rn_0$, we get $\chi(\varphi^{-rn_0-T_0}(z,\zeta))=0$. By (i), we know $r\in\{2,\ldots, j+1\}$. Thus, \eqref{eq:3.27} holds.

To show \eqref{eq:3.22}, we write
\[
A_+^{j+1}=\Op_h^{L_u}\left(\chi_1\left(1-\prod_{m=1}^{j+1}\chi\circ\varphi^{-mn_0}\right)\right)+\Op_h^{L_u}\left(\chi_2\left(1-\prod_{m=1}^{j+1}\chi\circ\varphi^{-mn_0}\right)\right).
\]
Writing the first term on the right using Lemma~\ref{obj:3.10} and \eqref{eq:3.27}, and the second term using Lemma~\ref{obj:3.11}, we obtain \eqref{eq:3.22}; note that \eqref{eq:3.21} is used to obtain the bound on $J_+^j$.

By \eqref{eq:3.19},
\[
\supp\big(\chi(1-\chi\circ\varphi^{-n_0})\big)\cap\Gamma_+\cap\{|\zeta|_g=1\}=\varnothing.
\]
Thus, by Lemma~\ref{obj:3.8}, we may write
\begin{equation}
A_+^1=Z_+^0(\omega)\mathcal P_h(\omega)+O(h^\infty)_{\mathcal D'\to\mathcal X},\label{eq:3.28}
\end{equation}
for $Z_+^0(\omega)$ holomorphic in $\omega\in\Omega$ and $\|Z_+^0(\omega)\|_{\mathcal Y\to\mathcal X}\leq Ch^{-1}$. To conclude, we set
\[
Z_+(\omega):=\sum_{j=0}^{k-1}J_+^{k-1}(\omega)\cdots J_+^{j+1}(\omega)Z_+^j(\omega)
\]
and use \eqref{eq:3.22},\eqref{eq:3.28}.
\end{proof}

Our final result of this section is analogous to \cite[Lemma 2.7]{Dy19a} and refines Lemma~\ref{obj:3.9}.

\begin{lem}\label{obj:3.13}
Fix $\rho\in[0,1)$, and $\chi\in C_c^\infty(T^*M\setminus0;[0,1])$ with $\chi=1$ near $K\cap\{|\zeta|_g=1\}$. Then there exists an integer $n_0>0$ such that uniformly in $\omega\in\Omega$,
\[
I=Z_-(\omega)\mathcal P_h(\omega)+J_-(\omega)\Op_h^{L_s}\left(\prod_{m=0}^k\chi\circ\varphi^{mn_0}\right)+O(h^\infty)_{\mathcal X\to\mathcal X}
\]
where $k=k(h)$ is the smallest integer such that $\frac\rho2\log\frac1h\leq kn_0$, and $Z_-(\omega),J_-(\omega)$ are holomorphic in $\omega\in\Omega$ and satisfy
\begin{equation}
\|Z_-(\omega)\|_{\mathcal Y\to\mathcal X}\leq Ch^{-1-\rho \nu_0},\qquad
\|J_-(\omega)\|_{H_h^{-N}\to\mathcal X}\leq C_N\exp(kn_0(\nu_0-\Im\omega/h)).\label{eq:3.29}
\end{equation}
\end{lem}

\begin{proof}
We follow the proof of \cite[Lemma 2.7]{Dy19a}, making necessary changes. Let $T_0,n_0,\chi_1,\chi_2$ be as in the proof of Lemma~\ref{obj:3.12}. Set
\[
A_-^j:=\Op_h^{L_s}\left(\prod_{m=0}^j\chi\circ\varphi^{mn_0}\right).
\]
With $k=k(h)$ as in the statement of the lemma, we claim that uniformly in $j=1,\ldots,k-1$,
\begin{equation}
A_-^j=Z_-^j(\omega)\mathcal P_h(\omega)+J_-^j(\omega)A_-^{j+1}+O(h^\infty)_{\mathcal D'\to\mathcal X},\label{eq:3.30}
\end{equation}
where $Z_-^j(\omega),J_-^j(\omega)$ are holomorphic in $\omega\in\Omega$ and satisfy the bounds \eqref{eq:3.23}. To see this, we first claim that for all $j=1,\ldots,k-1$,
\[
\varphi^{-(n_0+T_0)}\left(\supp\left(\chi_1\prod_{m=1}^j\chi\circ\varphi^{mn_0}\right)\right)\subset\left\{\prod_{m=0}^{j+1}\chi\circ\varphi^{mn_0}=1\right\}.
\]
Indeed, let $(z,\zeta)\in\supp\left(\chi_1\prod_{m=1}^j\chi\circ\varphi^{mn_0}\right)$. We need to show that
\[
\chi(\varphi^{mn_0-(n_0+T_0)}(z,\zeta))=1,\qquad m=0,\ldots,j+1.
\]
Applying \eqref{eq:3.24} gives the cases $m=0,1$, while for $m=2,\ldots,j+1$ we apply \eqref{eq:3.26} to $\varphi^{(m-1)n_0-T_0}(z,\zeta)$, $t_1=(m-1)n_0-T_0$, $t_2=T_0$. Thus, \eqref{eq:3.30} is proved using Lemmas~\ref{obj:3.10} and \ref{obj:3.11} as in the proof of Lemma~\ref{obj:3.12}.

Next, we have
\[
K\cap\{|\zeta|_g=1\}\subset\{\chi(\chi\circ\varphi^{n_0})=1\}.
\]
It follows from Lemma~\ref{obj:3.9} that
\begin{equation}
I=Z_-^0(\omega)\mathcal P_h(\omega)+J_-^0(\omega)A_-^1+O(h^\infty)_{\mathcal D'\to\mathcal X}\label{eq:3.31}
\end{equation}
where $Z_-^0(\omega),J_-^0(\omega)$ are holomorphic in $\omega\in\Omega$ and satisfy $\|Z_-^0(\omega)\|_{\mathcal Y\to\mathcal X}\leq Ch^{-1}$, $\|J_-^0(\omega)\|_{H_h^{-N}\to\mathcal X}\leq C_N$ for all $N$. The conclusion follows by setting
\[
Z_-(\omega):=\sum_{j=0}^{k-1}J_-^0(\omega)\cdots J_-^{j-1}(\omega)Z_-^j(\omega),\qquad J_-(\omega):=J_-^0(\omega)\cdots J_-^{k-1}(\omega)
\]
and using \eqref{eq:3.30},\eqref{eq:3.31}.
\end{proof}

Together, Lemmas~\ref{obj:3.12} and \ref{obj:3.13} show that resonant states inherit strong microlocalization properties. We exploit these results in the next section to prove our spectral gap.

\section{Reduction to the fractal uncertainty principle}\label{sec:reduction-fup}

In this section we use the fine microlocal estimates of the previous section to reduce the proof of our Theorem to a certain operator norm bound (Proposition~\ref{obj:4.1}). We then make this estimate by modifying part of the approach of \cite{ADM}, which estimates a similar operator via the Fractal Uncertainty Principle of Bourgain--Dyatlov \cite{BoDy18}, as generalized by Dyatlov--Jin--Nonnenmacher \cite{DJN} (see \cite[Section 5.2]{ADM}). Our work in Section~\ref{sec:scattering-resolvent} -- in particular Lemmas~\ref{obj:3.12} and \ref{obj:3.13} -- is inspired by the form of the operators analyzed in \cite[Section 5]{ADM}, and allows a straightforward adaptation of their methods to the present setting.

We begin by constructing the operators we use in our modification of the argument of \cite{DyZa}, \cite{Dy19a}. Recall the fast horocyclic flows $e^{sV^\pm}:S^*M\to S^*M$ from Section~\ref{sec:complex-hyperbolic-manifolds}. Under the hypothesis of our Theorem that $\Lambda_\Gamma$ contains no complex circle, Lemma~\ref{obj:2.5} shows that $K\cap\{|\zeta|_g=1\}$ may contain no $V^\pm$-orbit. Since $K\cap\{|\zeta|_g=1\}$ is compact, we may apply Lemma~\ref{obj:2.4} (twice) to find a set $U\subset S^*M$ such that
\begin{enumerate}[i)]
\item $S^*M\setminus U$ is a compact set containing $K\cap\{|\zeta|_g=1\}$ in its interior,
\item $U$ is $V^\pm$-dense,
\item There exists $T>0$ such that each $V^\pm$-segment of length $T$ intersects $U$.
\end{enumerate}
We extend $U$ to a conic open set in $T^*M\setminus0$, also denoted $U$.

We now take $\chi_\pm\in C_c^\infty(T^*M\setminus0;[0,1])$ with $\chi_\pm=1$ near $K\cap\{|\zeta|_g=1\}$, $\supp\chi_\pm\subset(T^*M\setminus U)^\circ$, and $\chi_+=1$ near $\supp\chi_-$.

Given small $\varepsilon_0>0$, we fix $\rho:=\frac23(1-\varepsilon_0)$ as in \cite{ADM}. Apply Lemma~\ref{obj:3.12} to $\chi_+$ to find an integer $n_0>0$ such that the symbol (recall that $k=k(h)$ is the smallest integer such that $\frac\rho2\log\frac1h\leq kn_0$),
\begin{equation}
a_+:=\prod_{m=0}^k\chi_+\circ\varphi^{-mn_0}\in S^{\rm comp}_{L_u,\rho+\varepsilon,\varepsilon}\label{eq:4.1}
\end{equation}
for small $\varepsilon>0$, satisfies
\begin{equation}
\Op_h^{L_u}(\chi_+)-\Op_h^{L_u}(a_+)=Z_+(\omega)\mathcal P_h(\omega)+O(h^\infty)_{\mathcal D'\to\mathcal X}.\label{eq:4.2}
\end{equation}
Similarly, we apply Lemma~\ref{obj:3.13} to $\chi_-$ to see that the symbol
\begin{equation}
a_-:=\prod_{m=0}^k\chi_-\circ\varphi^{mn_0}\in S^{\rm comp}_{L_s,\rho+\varepsilon,\varepsilon}\label{eq:4.3}
\end{equation}
for small $\varepsilon>0$, satisfies
\begin{equation}
I=Z_-(\omega)\mathcal P_h(\omega)+J_-(\omega)\Op_h^{L_s}(a_-)+O(h^\infty)_{\mathcal X\to\mathcal X}.\label{eq:4.4}
\end{equation}
The fact that we may take the same integer $n_0$ in these applications despite different cutoffs follows because we may take $T_0$ as large as we like in the proofs of Lemmas~\ref{obj:3.12} and \ref{obj:3.13}.

We now reduce the proof of the Theorem to an operator norm bound:

\begin{prop}\label{obj:4.1}
If there exists $\beta_0>0$ such that
\begin{equation}
\|\Op_h^{L_s}(a_-)\Op_h^{L_u}(a_+)\|_{L^2\to L^2}\leq Ch^{\beta_0},\label{eq:4.5}
\end{equation}
then there exist $C_0,\beta>0$ such that for small $\varepsilon>0$ and each $\chi\in C_c^\infty(M)$, we have
\begin{equation}
\|\chi R(\lambda)\chi\|_{L^2\to L^2}\leq C_{\chi,\varepsilon}|\lambda|^{-1-\frac43\min\{0,\Im\lambda\}+\varepsilon},\qquad |\lambda|>C_0,\quad \Im\lambda>-\beta.\label{eq:4.6}
\end{equation}
\end{prop}

\begin{proof}
Define $\beta:=\min\{\frac{\beta_0}{\rho},\frac{1}{\rho}\}$ and fix $\nu_0<\beta$. Similarly to the proofs of \cite[Theorem 3]{DyZa}, \cite[Theorem 2]{Dy19a}, and \cite[Theorem 1.2]{Cu}, it suffices to prove that for any $u\in\mathcal X$, we have
\begin{equation}
\|u\|_{\mathcal X}\leq Ch^{-1-\frac43\max\{0,\nu_0\}-\varepsilon}\|\mathcal P_h(\omega)u\|_{\mathcal Y},
\qquad \omega\in\Omega:=[1-2h,1+2h]+ih[-\nu_0,\tfrac14].\label{eq:4.7}
\end{equation}
Indeed, using \eqref{eq:3.4} and the fact that $H_h^\phi\subset H_h^{s_-}$ (see the Remark following Lemma~\ref{obj:3.8}), the bound \eqref{eq:4.7} implies an $H_h^{s_--1}\to H_h^{s_-}$ bound on $\psi(-h^2\Delta_g-\frac{h^2n^2}{4}-\omega^2)^{-1}\psi$ for any $\psi\in C_c^\infty(M)$. Using an elliptic parametrix for $-h^2\Delta_g-\frac{h^2n^2}{4}-\omega^2$ (see, e.g., \cite[Proposition E.32]{DyZw}), we convert this to an $L^2\to L^2$ estimate, which implies \eqref{eq:4.6} after rescaling.

To prove \eqref{eq:4.7}, we first use \cite[Section 4.2.1]{ADM} to see that since $\chi_+\in C_c^\infty(T^*M\setminus0;[0,1])$ is independent of $h$,
\[
\Op_h^{L_u}(\chi_+)=\Op_h(\chi_+)+O(h)
\]
where $\Op_h$ is the standard semiclassical quantization of \cite[Section 4.1]{ADM}. Moreover, since $\chi_+=1$ on $\supp\chi_-$, we have
\[
\Op_h(\chi_+)=I+O(h^\infty)\quad\text{microlocally near }\supp\chi_-.
\]
Thus, from \eqref{eq:4.2},
\begin{equation}
I=Z_+(\omega)\mathcal P_h(\omega)+\Op_h^{L_u}(a_+)+O(h)\quad\text{microlocally near }\supp\chi_-.\label{eq:4.8}
\end{equation}
Since $\Op_h^{L_s}(a_-)$ is pseudolocal and its wavefront set is contained in $\supp\chi_-$, we obtain from \eqref{eq:4.8} and \eqref{eq:4.4} that for any $u\in\mathcal X$, and any $\varepsilon>0$,
\[
\begin{split}
\|u\|_{\mathcal X}\leq{}&Ch^{-1-2\rho\nu_0-\varepsilon}\|\mathcal P_h(\omega)u\|_{\mathcal Y}
+Ch^{-\rho \nu_0-\varepsilon}\|\Op_h^{L_s}(a_-)\Op_h^{L_u}(a_+)\|_{\mathcal X\to L^2}\|u\|_{\mathcal X}\\
&+Ch^{1-\rho \nu_0-\varepsilon}\|u\|_{\mathcal X},
\end{split}
\]
where we used the bounds \eqref{eq:3.18}, \eqref{eq:3.29}, and the fact that $\exp(kn_0(\nu_0-\Im\omega/h))\leq Ch^{-\rho \nu_0-\varepsilon}$ for any $\varepsilon$. Thus, if \eqref{eq:4.5} holds, we use that $\mathcal X\subset H^{\phi}_{h}\subset H_h^{s_-}\subset L^2$ to see that
\[
\|\Op_h^{L_s}(a_-)\Op_h^{L_u}(a_+)\|_{\mathcal X\to L^2}\leq Ch^{\beta_0}.
\]
Then, recalling that $\rho=\frac23(1-\varepsilon_0)\leq \frac{2}{3}$, and using that $\varepsilon$ is arbitrary, we easily obtain \eqref{eq:4.7}.
\end{proof}

It remains to prove \eqref{eq:4.5}. Although our symbols $a_\pm$ defined in \eqref{eq:4.1}, \eqref{eq:4.3} are somewhat different than those used in \cite[Section 5]{ADM}, they are of a sufficiently similar form that we may apply their approach with only minor modifications. We very closely follow the exposition of \cite[Section 5]{ADM}, indicating necessary changes.

As in \cite[Section 5.1.2]{ADM} we begin by decomposing the operator in \eqref{eq:4.5}, but adapt this decomposition to our situation: Let $q_1,\ldots,q_L\in\supp\chi_-$ be a maximal $h^{\rho/2}$-separated set, meaning $d(q_\ell,q_{\ell'})\geq h^{\rho/2}$ for all $\ell\neq\ell'$ and the balls $B(q_\ell;h^{\rho/2})$ cover $\supp\chi_-$. We then construct an $h$-dependent partition of unity
\[
\psi_\ell\in C_c^\infty(T^*M),\qquad \supp\psi_\ell\subset B(q_\ell;2h^{\rho/2}),\qquad \sum_{\ell=1}^L\psi_\ell^2=1\quad\text{on }\supp\chi_-,
\]
with the functions $\psi_\ell$ satisfying the derivative bounds for all multiindices $\alpha$,
\[
\sup|\partial^\alpha\psi_\ell|\leq C_\alpha h^{-\rho|\alpha|/2}.
\]
Notice that then for small $\varepsilon>0$, the symbols $a_+\psi_\ell$ and $a_-\psi_\ell$ are in $S^{\rm comp}_{L_u,\rho+\varepsilon,\rho/2}$ and $S^{\rm comp}_{L_s,\rho+\varepsilon,\rho/2}$ respectively.

In exactly the same way now as in \cite[Section 5.1.2]{ADM}, we decompose the operator in \eqref{eq:4.5} using this partition of unity, and the bound in \eqref{eq:4.5} follows by proving that
\begin{equation}
\max_\ell\|\Op_h^{L_s}(a_-\psi_\ell)\Op_h^{L_u}(a_+\psi_\ell)\|_{L^2\to L^2}\leq Ch^{\beta_0}.\label{eq:4.9}
\end{equation}

Next we will follow \cite[Section 5.3]{ADM} to prove an analogue of \cite[Lemma 5.5]{ADM}. First, we recall the notion of $\nu$-porous set:

\begin{defi}[{\cite[Definition 5.2]{ADM}}]\label{obj:4.2}
Let $\nu\in(0,1)$ and $0<\alpha_0\leq\alpha_1$. We say that a subset $\Omega\subset\mathbb R$ is $\nu$-porous on scales $\alpha_0$ to $\alpha_1$ if for each interval $I\subset\mathbb R$ of length $|I|\in[\alpha_0,\alpha_1]$, there exists a subinterval $J\subset I$ of length $|J|=\nu|I|$ such that $J\cap\Omega=\varnothing$.
\end{defi}

The importance of $\nu$-porous sets is that they satisfy the Fractal Uncertainty Principle (FUP), see \cite[Section 5.2]{ADM} for the version relevant to this paper. We do not review the FUP in any detail here, and we use it only indirectly at the end of our proof by concluding via the method of \cite[Section 5.5]{ADM}, which applies it in an essential way to prove a bound similar to \eqref{eq:4.5}.

We now fix $\ell$ and let $q_\ell$ be the corresponding point chosen in the construction of the partition of unity above. Let $\kappa_\ell:U_\ell\to T^*\mathbb R^{2n}$ be the symplectomorphism of Lemma~\ref{obj:2.1} applied with $q^0=q_\ell$. In particular, we recall from \eqref{eq:2.12}, \eqref{eq:2.13}, \eqref{eq:2.14} that
\[
\kappa_\ell(q_\ell)=0,\qquad d\kappa_\ell(q_\ell)V_\perp^+(q_\ell)=\ker d y_1,\qquad d\kappa_\ell(q_\ell)V_\perp^-(q_\ell)=\ker d\eta_1.
\]
Recall that $\rho=\frac23(1-\varepsilon_0)$ is fixed. We shall prove the following, whose statement is similar to \cite[Lemma 5.5]{ADM} but of course our symbols are different.

\begin{lem}\label{obj:4.3}
There exist sets $\Omega_\pm\subset\mathbb R$ such that
\begin{equation}
\kappa_\ell(\supp(a_+\psi_\ell))\subset\{(y,\eta):y_1\in\Omega_+\}\label{eq:4.10}
\end{equation}
\begin{equation}
\kappa_\ell(\supp(a_-\psi_\ell))\subset\{(y,\eta):\eta_1\in\Omega_-\}\label{eq:4.11}
\end{equation}
and the sets $\Omega_\pm$ are $\nu$-porous on scales $C_0h^\rho$ to $1$ for constants $\nu>0$ and $C_0$ which depend only on $(M,g)$, the (uniform in $\ell$) bounds on the derivatives of the $\kappa_\ell$, and the cutoffs $\chi_\pm$ defining the $a_\pm$, and in particular do not depend on $\ell$.
\end{lem}

We will prove \eqref{eq:4.11}, with \eqref{eq:4.10} proved analogously. Recall that
\begin{equation}
a_-=\prod_{m=0}^k\chi_-\circ\varphi^{mn_0},\label{eq:4.12}
\end{equation}
with $\chi_-\in C_c^\infty(T^*M\setminus0;[0,1])$, $\chi_-=1$ on $K\cap\{|\zeta|_g=1\}$. Moreover, $\supp\chi_-\subset(T^*M\setminus U)^\circ$ where $U$ is the conic open set in $T^*M\setminus0$ constructed using Lemma~\ref{obj:2.4} at the beginning of the section. Without loss, we may assume that $\supp\chi_-\subset\{\frac12\leq|\zeta|_g\leq2\}$.

Let $V\subset T^*M\setminus0$ be a closed, conic set with $\supp\chi_-\subset V^\circ$ and $V\cap\overline U=\varnothing$. Then let $V'$ be an open conic set with $V\subset V'$, $\overline{V'}\cap\overline U=\varnothing$. Note that from \eqref{eq:4.12} this implies
\begin{equation}
\supp a_-\subset\left(\bigcap_{m=0}^k\varphi^{-mn_0}(V)\right)\cap\left\{\frac14\leq|\zeta|_g\leq4\right\}.\label{eq:4.13}
\end{equation}
As in \cite[pg. 51]{ADM}, we now fix a large constant $C_1$ such that (note (1), (2), (3) below are the same as \cite[pg. 51]{ADM}, but (4) is adapted to our situation):
\begin{enumerate}[(1)]
\item We have
\begin{equation}
\supp\psi_\ell\subset\kappa_\ell^{-1}(\{(y,\eta):|y|+|\eta|\leq C_1h^{\rho/2}\}).\label{eq:4.14}
\end{equation}
\item We have the upper bounds on the derivatives of the trajectory $s\mapsto\kappa_\ell(e^{sV^-}(q_\ell))$
\begin{equation}
|\partial_s y(\kappa_\ell(e^{sV^-}(q_\ell)))|+|\partial_s\eta(\kappa_\ell(e^{sV^-}(q_\ell)))|\leq C_1,\qquad s\in[-C_1^{-1},C_1^{-1}],\label{eq:4.15}
\end{equation}
\[
|\partial_s^2\eta_1(\kappa_\ell(e^{sV^-}(q_\ell)))|\leq C_1,\qquad s\in[-C_1^{-1},C_1^{-1}].
\]
\item We have the lower bound on the derivative of the $\eta_1$-component of the above trajectory:
\begin{equation}
|\partial_s\eta_1(\kappa_\ell(e^{sV^-}(q_\ell)))|\geq C_1^{-1},\qquad s\in[-C_1^{-1},C_1^{-1}].\label{eq:4.16}
\end{equation}
\item The distance between $V\cap\{\frac14\leq|\zeta|_g\leq4\}$ and $T^*M\setminus V'$ is at least $C_1^{-1}$:
\begin{equation}
q\in V\cap\{\tfrac14\leq|\zeta|_g\leq4\},\quad d(q,q')\leq C_1^{-1}\quad\Longrightarrow\quad q'\in V'.\label{eq:4.17}
\end{equation}
\end{enumerate}

We now define analogues of $\widetilde\Omega_-,\Omega_-$ from \cite[(5.26), (5.27)]{ADM}: Let $C_2$ be a large integer independent of $h$. We choose a more precise $C_2$ in a moment. Set
\begin{equation}
\widetilde\Omega_-:=\left\{s\in[-C_1^{-1},C_1^{-1}]:e^{sV^-}(q_\ell)\in\bigcap_{j=0}^{k-C_2}\varphi^{-jn_0}(V')\right\}.\label{eq:4.18}
\end{equation}
Then we take
\begin{equation}
\Omega_-:=\eta_1\left(\kappa_\ell\left(\{e^{sV^-}(q_\ell):s\in\widetilde\Omega_-\}\right)\right)\cap[-C_1h^{\rho/2},C_1h^{\rho/2}].\label{eq:4.19}
\end{equation}
Our Lemma~\ref{obj:4.3} follows from Lemmas~\ref{obj:4.4} and \ref{obj:4.5} below, which are analogues of \cite[Lemmas 5.6 and 5.7]{ADM} respectively. The proofs are very similar to their analogues in \cite{ADM}, and we will repeat much of them with minimal changes, but include complete proofs here to make clear the somewhat subtle modifications needed to adapt them to our setting.

We begin with the following analogue of \cite[Lemma 5.6]{ADM}:

\begin{lem}\label{obj:4.4}
For $C_2$ large enough depending only on $(M,g)$, the derivative bounds for the $\kappa_\ell$, and $C_1$, the inclusion \eqref{eq:4.11} holds.
\end{lem}

\begin{proof}
1. We follow the proof of \cite[Lemma 5.6]{ADM}. It suffices to show that for arbitrary $\widetilde q\in\supp(a_-\psi_\ell)$, we have
\begin{equation}
\widetilde\eta_1\in\Omega_-\qquad\text{where}\qquad \widetilde\eta_1=\eta_1(\kappa_\ell(\widetilde q)).\label{eq:4.20}
\end{equation}
To see this, first note that $|\widetilde\eta_1|\leq C_1h^{\rho/2}$ by \eqref{eq:4.14}, and thus for small enough $h$ depending on $C_1$, we have $|\widetilde\eta_1|\leq C_1^{-2}$. From \eqref{eq:4.16} and using that $\kappa_\ell(q_\ell)=0$, there exists $s\in\mathbb R$ such that $|s|\leq C_1^2h^{\rho/2}\leq C_1^{-1}$ and
\[
\eta_1(\kappa_\ell(e^{sV^-}(q_\ell)))=\widetilde\eta_1.
\]
To prove \eqref{eq:4.20}, it suffices to show that $s\in\widetilde\Omega_-$.

2. Using \eqref{eq:4.14} and \eqref{eq:4.15}, we see that $\widetilde q$ and $e^{sV^-}(q_\ell)$ lie in
\[
R^-:=\kappa_\ell^{-1}(\{(y,\eta):|y|+|\eta|\leq C_1^3h^{\rho/2},\ \eta_1=\widetilde\eta_1\}).
\]
Using Lemma~\ref{obj:2.2} with $\alpha=C_1^3h^{\rho/2}$, we find $C_3>0$ depending only on $(M,g)$ and the derivative bounds on the $\kappa_\ell$ such that for all $t\geq0$ with $C_1^3h^{\rho/2}e^t<1$, we have
\[
d(\varphi^t(\widetilde q),\varphi^t(e^{sV^-}(q_\ell)))\leq C_3C_1^3h^{\rho/2}e^t.
\]
Take $C_2$ large enough so that
\[
C_3C_1^3e^{(1-C_2)n_0}\leq C_1^{-1}.
\]
Then for any $j\in\{0,1,\ldots,k-C_2\}$, recalling that $\rho=\frac23(1-\varepsilon_0)$ and the definition of $k$ as the least integer such that $\frac\rho2\log\frac1h\leq kn_0$ (and hence $kn_0\leq\frac\rho2\log\frac1h+n_0$), we have
\begin{equation}
d(\varphi^{jn_0}(\widetilde q),\varphi^{jn_0}(e^{sV^-}(q_\ell)))\leq C_3C_1^3e^{(1-C_2)n_0}\leq C_1^{-1}.\label{eq:4.21}
\end{equation}
But $\varphi^{jn_0}(\widetilde q)\in V\cap\{\frac14\leq|\zeta|_g\leq4\}$ by \eqref{eq:4.13}, so by \eqref{eq:4.21} and \eqref{eq:4.17}, we see that $\varphi^{jn_0}(e^{sV^-}(q_\ell))\in V'$. It follows that $s\in\widetilde\Omega_-$ and the proof is complete.
\end{proof}

Next we prove the analogue of \cite[Lemma 5.7]{ADM}:

\begin{lem}\label{obj:4.5}
The set $\Omega_-$ defined in \eqref{eq:4.19} is $\nu$-porous on scales $C_0h^\rho$ to $1$ for some constants $\nu>0$ and $C_0$ which depend only on the sets $V',U$, and the constants $C_1,C_2$.
\end{lem}

\begin{proof}
1. We follow the proof of \cite[Lemma 5.7]{ADM}. Recall that $U$ was defined via Lemma~\ref{obj:2.4} for both $V^\pm$ and then extended to an open conic set in $T^*M\setminus0$. In particular, $U\cap S^*M$ is $V^-$-dense. By Lemma~\ref{obj:2.4} (iii), there exists $T\geq1$ such that each $V^-$-segment of length $T$ in $S^*M$ intersects $U$. Since $V'\cap\overline U=\varnothing$ and $S^*M\setminus U$ is compact, there exists $\delta>0$ such that each $V^-$-segment of length $T$ in $S^*M$ has a subsegment of length $\delta$ that does not intersect $V'$. Since the vector field $V^-$ is extended homogeneously from $S^*M$ to $T^*M\setminus0$ and $V'$ is a conic set, the previous statements extend to all $V^-$-segments of length $T$ in $T^*M\setminus0$.

We define constants
\[
\nu':=e^{-2n_0}T^{-1}\delta,\qquad C_0':=e^{2n_0(C_2+1)}T.
\]
2. Next we show that the set $\widetilde\Omega_-$ from \eqref{eq:4.18} is $\nu'$-porous on scales $C_0'h^\rho$ to $1$. We note that from \eqref{eq:2.10} it follows that for each $t\in\mathbb R$, the image under $\varphi^t$ of a $V^-$-segment of length $\alpha$ is a $V^-$-segment of length $e^{2t}\alpha$.

Fix an interval $I\subset\mathbb R$ of length $|I|\in[C_0'h^\rho,1]$, and choose $b\in\mathbb Z$ such that
\[
T\leq e^{2bn_0}|I|\leq e^{2n_0}T.
\]
Then $b\geq0$ since $|I|\leq1\leq T$, and moreover we have $C_0'h^\rho\leq|I|\leq e^{2n_0-2bn_0}T$. Recalling the definition of $k=k(h)$ as the smallest integer such that $\frac\rho2\log\frac1h\leq kn_0$, we see that
\[
b\leq\frac\rho{2n_0}\log\frac1h-C_2\leq k-C_2.
\]
The set $\Gamma_I:=\{e^{sV^-}(q_\ell):s\in I\}$ is a $V^-$-segment in $T^*M\setminus0$ of length $|I|$. By the remark above, then, $\varphi^{bn_0}(\Gamma_I)$ is a $V^-$-segment of length $e^{2bn_0}|I|\geq T$. Using part 1 of the proof, there exists a subsegment of $\varphi^{bn_0}(\Gamma_I)$ of length $\delta$ that does not intersect $V'$. This subsegment can be written as $\varphi^{bn_0}(\Gamma_J)$ with $\Gamma_J=\{e^{sV^-}(q_\ell):s\in J\}$ and $J\subset I$ a subinterval of length
\[
|J|=e^{-2bn_0}\delta\geq\nu'|I|.
\]
From this construction, we see that $\varphi^{bn_0}(e^{sV^-}(q_\ell))\notin V'$ for all $s\in J$. By \eqref{eq:4.18}, this shows that $J\cap\widetilde\Omega_-=\varnothing$, finishing the proof of porosity of $\widetilde\Omega_-$.

3. The porosity of the set $\Omega_-$ now also follows as in the proof of \cite[Lemma 5.7]{ADM}. Define $\psi(s):=\eta_1(\kappa_\ell(e^{sV^-}(q_\ell)))$ for $|s|\leq C_1^{-1}$. By \eqref{eq:4.15} and \eqref{eq:4.16}, $\psi$ extends to a diffeomorphism of $\mathbb R$ satisfying
\[
\max(\sup|\psi'|,\sup|{\psi'}|^{-1},\sup|\psi''|)\leq2C_1.
\]
Then \eqref{eq:4.19} shows that $\Omega_-\subset\psi(\widetilde\Omega_-)$. Porosity of $\widetilde\Omega_-$ along with \cite[Lemma 2.12]{DJN} prove that $\Omega_-$ is $\nu$-porous on scales $C_0h^\rho$ to $\alpha_1$ with
\[
\nu:=\frac12\nu',\qquad C_0:=2C_1C_0',\qquad \alpha_1:=\frac12C_1^{-3}.
\]
Finally, as $\Omega_-\subset[-C_1h^{\rho/2},C_1h^{\rho/2}]$, it follows that $\Omega_-$ is also $\nu$-porous on scales $\alpha_1$ to $1$ if $h$ is small enough depending on $C_1$.
\end{proof}

Having proved the porosity of the projection of the supports of the symbols $a_\pm\psi_\ell$ under the symplectomorphism $\kappa_\ell$ (Lemma~\ref{obj:4.3}), the remainder of the proof of \eqref{eq:4.5} now follows exactly as in \cite[Section 5.5]{ADM}: We conjugate the operators $\Op_h^{L_s}(a_-\psi_\ell)$ and $\Op_h^{L_u}(a_+\psi_\ell)$ by the semiclassical Fourier integral operators constructed in \cite[Section 5.4]{ADM} and this localizes the operators to porous sets in both position $y_1$ and frequency $\eta_1$. The bound \eqref{eq:4.5} and hence the conclusion of our Theorem then follow from the Fractal Uncertainty Principle of \cite{BoDy18} as generalized by \cite{DJN}. We refer to \cite[Section 5.5]{ADM} for the details, and conclude the proof.

\printbibliography

\end{document}